\documentclass[english]{amsart}
\usepackage{geometry}
\usepackage{esint}
\usepackage{hyperref}
\usepackage{mathtools}
\usepackage{longtable}

\usepackage{amssymb,amsmath,amsthm,amscd}
\usepackage{mathrsfs}
\usepackage{tikz}
\usetikzlibrary{positioning}
\usepackage{graphics}
\usetikzlibrary{matrix,arrows}
\usepackage{amssymb}
\usepackage{babel}
\usepackage{amstext}
\usepackage{amscd}
\usepackage{epsfig}
\usepackage{rotating}
\usepackage{psfrag}
\usepackage{multicol}
\usepackage{multirow}
\usepackage{array}
\usepackage{verbatim}
\usepackage{enumitem}
\usepackage{colortbl}
\usepackage{hhline}
\usepackage{xcolor}
\usepackage[abs]{overpic}
\usepackage{minibox}
\usepackage{enumitem}
\usepackage{mathtools}

\makeatletter
\numberwithin{equation}{section}
\numberwithin{figure}{section}

\theoremstyle{plain}
\newtheorem{thm}{Theorem}
  \theoremstyle{plain}
  \numberwithin{thm}{section}
  \newtheorem{cor}[thm]{Corollary}
  \theoremstyle{plain}
  \newtheorem{lem}[thm]{Lemma}
   
  \theoremstyle{remark}
  \newtheorem{rem}[thm]{Remark}
    \theoremstyle{remark}

   \theoremstyle{plain}
  \newtheorem{prop}[thm]{Proposition}
  \def\Ddots{\mathinner{\mkern1mu\raise\p@
\vbox{\kern7\p@\hbox{.}}\mkern2mu
\raise4\p@\hbox{.}\mkern2mu\raise7\p@\hbox{.}\mkern1mu}}
\makeatother

\newtheorem{definition}{Definition}
\numberwithin{definition}{section}

   \theoremstyle{remark}
\newtheorem*{rem*}{Remark}

\newcounter{anhang}

\newcommand{\norm}[1]{\left\| #1 \right\|}
\newcommand{\mklm}[1]{\left\{ #1 \right\}}
\newcommand{\eklm}[1]{\left\langle #1 \right\rangle}

\renewcommand{\d}{\,d}

\newcommand{\N}{{\mathbb N}}
\newcommand{\Z}{{\mathbb Z}}
\newcommand{\C}{{\mathbb C}}

\newcommand{\R}{{\mathbb R}}

\newcommand{\D}{{\mathcal D}}
\newcommand{\E}{{\mathcal E}}
\newcommand{\F}{{\mathcal F}}

\newcommand{\I}{{\mathcal I}}
\newcommand{\J}{{ \mathcal J}}
\newcommand{\M}{{\mathscr M}}

\renewcommand{\O}{{\mathcal O}}

\newcommand{\W}{{\mathcal W}}

\renewcommand{\epsilon}{\varepsilon}

\newcommand{\zetaF}{\zeta}

\newcommand{\bdm}{\begin{displaymath}}
\newcommand{\edm}{\end{displaymath}}
\newcommand{\bq}{\begin{equation}}
\newcommand{\eq}{\end{equation}}
\newcommand{\bqn}{\begin{equation*}}
\newcommand{\eqn}{\end{equation*}}

\newcommand{\Cinft}{{\rm C^{\infty}}}
\newcommand{\CT}{{\rm C^{\infty}_c}}

\renewcommand{\S}{{\mathcal S}}

\newcommand{\SO}{\mathrm{SO}}

\newcommand{\g}{{\bf \mathfrak g}}
\newcommand{\h}{{\bf \mathfrak h}}

\renewcommand{\t}{{\bf \mathfrak t}}

\newcommand{\eps}{\varepsilon}

\newcommand{\Ad}{\mathrm{Ad}\,}

\renewcommand{\Im}{\mathrm{Im}\,}
\renewcommand{\Re}{\mathrm{Re}\,}
\newcommand{\vol}{\mathrm{vol}\,}

\newcommand{\Crit}{\mathrm{Crit}}

\DeclareMathOperator{\supp}{supp\,}

\DeclareMathOperator{\gd}{\partial}

\begin{document}
\title[Asymptotic expansion of generalized Witten integrals]{Asymptotic expansion of generalized Witten integrals\\ for Hamiltonian circle actions}

\author[B. K\"uster]{Benjamin K\"uster}
\email{bkuester@math.upb.de}
\address{Universit\"at Paderborn, Institut f\"ur Mathematik, Warburger Str.\ 100, 33098 Paderborn, Germany}

\author[P. Ramacher]{Pablo Ramacher}
\email{ramacher@mathematik.uni-marburg.de}
\address{Philipps-Universit\"at Marburg, Fachbereich Mathematik und Informatik, Hans-Meer\-wein-Str., 35032 Marburg, Germany}

\date{\today}

\begin{abstract}
We derive a complete asymptotic expansion of generalized Witten integrals for Hamiltonian circle actions on arbitrary symplectic manifolds, characterizing the coefficients in the expansion as integrals over the symplectic strata of the corresponding Marsden-Weinstein reduced space and distributions on the Lie algebra.  The obtained coefficients involve singular contributions of the lower-dimensional strata related to numerical invariants of the fixed-point set.
\end{abstract}

\maketitle

\tableofcontents

\section{Introduction}

 Let $(M,\omega)$ be a $2n$-dimensional connected  symplectic manifold with a non-trivial  Hamiltonian action of a compact connected Lie group $G$ and momentum map $\J:M \rightarrow \g^\ast$, \label{Jintrod} where $\g^\ast$ denotes the dual of the Lie algebra $\g$ of $G$.  If $\zeta \in \g^\ast$ is a regular value of $\J$,  the corresponding \emph{Marsden-Weinstein symplectic quotient} or \emph{reduced space}
 $
 \M^\zeta:= \J^{-1}(\mklm{\zeta})/G
 $ 
 is a symplectic orbifold,  and if $\zeta$ is not a regular value,  $\M^\zeta$ is a stratified space which can have serious  singularities. The  geometry and topology  of $\M^\zeta$  have been extensively studied in the last decades \cite{lerman-sjamaar,kirwan84,witten92,kalkman95,lerman-tolman00,JKKW03} mostly for compact $M$, a major tool being  the Witten integral  and its asymptotic expansion, which carries important geometric and topological information.  
 
In this paper, we study Witten-type integrals in the case where $G=T:=\SO(2)\cong S^1$ is the circle group and $\zeta \in \t^\ast:=\g^\ast$ is not necessarily a regular value.  While we do not assume $M$ to be compact, we  consider compactly supported integrands. More precisely, we derive  a complete asymptotic expansion of \emph{generalized Witten integrals}  of the form
  \bq
\label{eq:2}
I_{a,\sigma}^\zeta(\eps):=\int_{\t}  \int_{ M } e^{i(\J(p)-\zeta)(x)/\eps} a(p) \d M(p)\sigma(x) \d x
\eq 
 for arbitrary $\zeta \in \t^\ast$ in integer powers of $\eps>0$, where  $\t:=\mathfrak{so}(2)$ is the Lie algebra of $T$ for which we fix an identification $\t\cong \R$, $a \in \CT(M)$ a compactly supported amplitude, $\sigma\in \S(\t)$ a Schwartz function on $\t$, $dM:=\omega^n/n!$ the symplectic volume form on $M$, and $dx$ the Lebesgue measure on $\t\cong \R$.  
 
 We regard 
 \bq
 I_{a,\sigma}^\zeta(\eps)=I^\zeta(\eps) ( a \otimes \sigma)\label{eq:Iasigmaeps}
 \eq
  as the evaluation of a distribution   $I^\zeta(\eps) \in (\D(M)\otimes \S(\t))'$ at the test function $a\otimes\sigma\in \D(M)\otimes \S(\t)$.  By \emph{complete asymptotic expansion} we mean an expansion of $I^\zeta(\eps)$ in  $(\D(M)\otimes \S(\t))'$ of the form
\bq
I^\zeta(\eps)\;\sim \;\eps^{j_0(\zeta)}\sum_{j=0}^{\infty} \eps^j A^\zeta_j,\qquad A^\zeta_j\in \D'(M)\otimes\S'(\t),\quad j_0(\zeta)\in \Z,\label{eq:exp000}
\eq
that is, for each $J_0 \in \N_0$ and each compact set $K\subset M$ there is an $N_{J_0,K,\zeta}\in \N_0$ and a family of differential operators $\{D_{J_0,K,\zeta}^l\}_{0\leq l \leq N_{J_0,K,\zeta}}$ on $M$ such that for all $\sigma\in \S(\t)$ and all $a\in \CT(M)$ with $\supp a\subset K$ one has
\bqn
\Big|I_{a,\sigma}^\zeta(\eps)-\eps^{j_0(\zeta)}\sum_{j=0}^{J_0} \eps^j A^\zeta_j(a\otimes\sigma)\Big|\leq \sum_{k,l=0}^{N_{J_0,K,\zeta}}\norm{D_{J_0,K,\zeta}^la}_\infty \big\Vert\sigma^{(k)}\big\Vert_\infty\eps^{j_0(\zeta)+J_0+1}.
\eqn
Here $\D(M)$ denotes the space of test functions on $M$, given by $\CT(M)$ with the test function topology; its dual $\D'(M)$ is the space of distributions\footnote{In this paper, we shall identify distributions with distribution densities on $M$ via the symplectic volume form $dM$, which defines a strictly positive smooth density on $M$.}   
 on $M$, and $\D'(M)\otimes\S'(\t)$ embeds into $(\D(M)\otimes \S(\t))'$. We are specially interested in the dependence of the coefficient distributions $A^\zeta_j$ and  the leading order
$$
\ell(\zeta):=j_0(\zeta)+\inf \big \{\, j\in \N_0 \mid  A_j^\zeta\neq 0\big \}
$$  on the parameter $\zeta\in \t^\ast$, which may be a  regular or singular value of $\J$. 

Distributions of the form $I^\zeta(\eps)$ arise in the study of the  Fourier transform  of \emph{Duistermaat-Heckman-type distributions}. The latter are tempered distributions  $L_\varrho\in \S'(\t)$ associated with a compactly supported equivariant differential form $\varrho$ on $M$, that is a polynomial map $\varrho:\t \to \Omega_c(M)^T$, by
\bq
L_\varrho(x):=\int_{M} e^{i(J(x)-\omega)} \varrho(x),\quad x\in \t,\qquad J(x)(p):=\J(p)(x),\quad p\in M.\label{eq:duistermaatheckman}
\eq
The connection between the integrals \eqref{eq:2} and the distributions  \eqref{eq:duistermaatheckman} is explained in detail in Section \ref{subsec:2.2}, where we also describe how the original integral studied by Witten \cite{witten92} arises as a special case of the generalized Witten integral \eqref{eq:2}. If $M$ is compact and $\varrho=1$, $L_\varrho$ corresponds to the inverse Fourier transform of the pushforward $\J_\ast(dM)$ of the symplectic volume form along $\J$. As was discovered by Duistermaat and Heckman  \cite{duistermaat-heckman82}, $\J_\ast(dM)$ is a piecewise polynomial measure on $\t^\ast$, or equivalently, $L_\varrho$ is exactly given by the leading term in the stationary phase approximation. This can be seen as a special instance of the localization formula of Berline-Vergne \cite{berline-vergne83, berline-getzler-vergne} and Atiyah-Bott \cite{atiyah-bott84}, one of the central  principles in equivariant cohomology.    For general $M$ and $\varrho$,  one would  expect  that the coefficients $A^\zeta_j$ in the expansion of $I^\zeta(\eps)$ are given by piecewise polynomial measures on $\t^\ast$ as well, and our results show that this is indeed the case.  Furthermore, as will be discussed below, these measures have a geometric meaning  in terms of the symplectic data of  $\M^\zeta$.

 If $\zeta \in \t^\ast$ is  a regular value of $\J$, the phase function $\psi^\zeta(p,x):=\J(p)(x)-\zeta(x)$  in \eqref{eq:2} is a \emph{Morse-Bott function} and the stationary phase theorem yields an expansion \eqref{eq:exp000}  of $I^\zeta(\eps)$ with $j_0(\zeta)=\ell(\zeta)=1$ and
 \bq
 A^\zeta_{0}(a\otimes\sigma)=2\pi \sigma(0)\int_{\mathscr M^\zeta}\eklm{a}_T\d \mathscr M^\zeta,\qquad A^\zeta_{j}(a\otimes\sigma)=\sigma^{(j)}(0)\int_{\mathscr M^\zeta}\big<D^\zeta_{j} a\big>_T\d \mathscr M^\zeta,\label{eq:regcoeff}
 \eq
 where $D^\zeta_{j}$ is a differential operator of order $j$ defined near   $\J^{-1}(\{\zeta\})$, $D^\zeta_{0}$ is just multiplication by $2\pi$, $\d \mathscr M^\zeta$ is the symplectic volume form on the orbifold $\M^\zeta=\J^{-1}(\{\zeta\})/T$, 
 and $\eklm{f}_T(T\cdot p):=\int_Tf(g\cdot p)\d g$ denotes the function on $M/T$ defined by integrating  $f\in \CT(M)$ over an orbit $T\cdot p\subset M$ using the Haar measure $dg$ on $T$ fixed by our identification $\t\cong \R$. Furthermore, $D^\zeta_{j}$ is  transversal to $\J^{-1}(\{\zeta\})$, and  the coefficients in \eqref{eq:regcoeff} depend smoothly on $\zeta$ in the sense that if $\mathscr V\subset \t^\ast$ is an open set consisting entirely of regular values of $\J$, then the function $\mathscr V\owns \zeta \to A^\zeta_{j}(a\otimes\sigma)\in \C$ is smooth for each $j$, $a$, and $\sigma$, see Proposition \ref{prop:asymptreg} for more details.
 
When $\zeta$ is a singular value of $\J$, serious difficulties arise in the study of the integrals \eqref{eq:2}, since then the stationary phase principle cannot be applied. Moreover, the behavior of the coefficients in the regular expansion \eqref{eq:regcoeff} as $\zeta$ approaches a singular value is  unclear a priori. In this paper, we address both of these  problems.  
In order to state our results, consider for an arbitrary $\zeta\in\t^\ast$ the 
stratification of the symplectic quotient $\M^\zeta=\J^{-1}(\{\zeta\})/T$ by infinitesimal orbit types
 \bq
\M^\zeta=  \M^\zeta_\mathrm{top} \sqcup  \M^\zeta_\mathrm{sing}, \qquad \M^\zeta_\aleph:=(\J^{-1}(\mklm{\zeta}) \cap M_{(\h_\aleph)})/T,\label{eq:mstrat}
\eq
 where $M_{(h_\aleph)}$ denotes the stratum of $M$ of infinitesimal orbit type $(\h_\aleph)$ with $\h_\mathrm{top}=\mklm{0}$ and  $\h_\mathrm{sing}=\t$.  $\M^\zeta_\mathrm{top}$  is  an orbifold called the \emph{top stratum}. It is either dense in $\M^\zeta$ or empty, which happens iff $T$ acts trivially on $\J^{-1}(\mklm{\zeta}$). The orbifold $\M^\zeta_\mathrm{top}$ inherits a symplectic form  $\omega^\zeta_\mathrm{top}$  uniquely characterized by $i^\ast \omega=\pi^\ast \omega^\zeta_\mathrm{top}$, where $i:\J^{-1}(\mklm{\zeta}) \cap M_{(\h_\mathrm{top})}\to M$ is the inclusion and $\pi:\J^{-1}(\mklm{\zeta}) \cap M_{(\h_\mathrm{top})} \rightarrow  \M^\zeta_\mathrm{top}$ the orbit projection. Writing $M^T$ for the space of fixed-points of the $T$-action on $M$ and $\F$ for the set of all connected components of $M^T$, \label{Fintrod}
 each $F\in \F$ is a symplectic submanifold of $(M,\omega)$ on which $\J$ is constant, and the singular values of $\J$ are    $\mklm{\J(F): F \in \F}\subset \t^\ast$. The space $\M^\zeta_\mathrm{sing}$ can be identified with  the union of all $F\in \F$ with $\J(F)=\zeta$. 
  Each $F\in \F$ provides certain numerical invariants of the Hamiltonian $T$-space $(M,\omega,\J)$.  The simplest is the codimension of $F$ in $M$, an even number denoted by  $\mathrm{codim}\,F$, which is non-zero thanks to our assumptions that the $T$-action on $M$ is non-trivial and that $M$ is connected. Moreover, the behavior of $\J$ near $F$ intrinsically determines two non-negative even integers $n^\pm_F$ fulfilling $n^+_F+n^-_F=\mathrm{codim}\,F$. Technically, $n^+_F$ and $n^-_F$ arise as the positive and negative indices of inertia of some non-degenerate quadratic form $Q_F$ on $\R^{\mathrm{codim}\,F}$ assigned to $F$, see Section \ref{subsec:4.1}.  We therefore call $F$ \emph{definite} with sign $s_F\in \{+,-\}$ if $n^{s_F}_F=\mathrm{codim}\,F$ and \emph{indefinite} otherwise. With these preparations, we can state our first main result, proved in Section \ref{sec:proof}. 
\begin{thm}\label{thm:main1}
For each $\zeta\in \t^\ast\cong \R$, the generalized Witten integral \eqref{eq:2} has an asymptotic expansion
$$I^\zeta(\eps)\;\sim\;\eps\sum_{j=0}^{\infty} \eps^j A^\zeta_j
$$
in $(\D(M)\otimes\S(\t))'$ with coefficient distributions  of the form
 $$
 A^\zeta_j=(A^\zeta_j)_\mathrm{top} + (A^\zeta_j)_\mathrm{sing}
 $$
 given by
 \bq
 (A^\zeta_j)_\mathrm{top}(a\otimes\sigma)=\sigma^{(j)}(0)\int_{\mathscr M^\zeta_{\mathrm{top}}}\big<D^\zeta_{j} a\big>_T\d \mathscr M^\zeta_{\mathrm{top}},\qquad 
 (A^\zeta_j)_\mathrm{sing}(a\otimes\sigma)=\sum_{\substack{F\in \F:\J(F)=\zeta,\\F\cap \supp a\neq \emptyset}} A_{j,F}(a\otimes\sigma),\label{eq:distrib1}
 \eq
 where $A_{j,F}=0$ unless $j\geq \frac{1}{2}\mathrm{codim}\, F-1$, in which case one has 
 \bqn
 A_{j,F}(a\otimes\sigma)=\begin{dcases} 
 \sigma^{[j]}_{\mp}(0)\int_{F}D_{j,F} a\d F, & F\text{ definite},\; s_F=\pm,\\
  \sigma^{[j]}_{+}(0)\int_{F}D^+_{j,F} a\d F+\sigma^{[j]}_{-}(0)\int_{F}D^-_{j,F} a\d F, &F\text{ indefinite}.
 \end{dcases}
 \eqn
 The objects occurring here are as follows:
\begin{itemize}[leftmargin=*]
 \item  $\S(\t)\owns\sigma\mapsto \sigma^{[j]}_{\pm}(0)\in \C$ are tempered distributions defined by
 \bq
 \sigma^{[j]}_{\pm}(0):=\frac {i^j}{2\pi} \eklm{ \xi^j_\pm, \hat \sigma}:=\frac{(\pm i)^j}{2\pi}\int_0^\infty\hat\sigma(\pm \xi)\xi^{j}\d \xi,\qquad j\in \N_0,\label{eq:distrib0}
 \eq 
 where we use our identification $\R\cong \t^\ast$ and  $\hat \sigma$ is the Fourier transform of $\sigma$, normalized such that $$\sigma^{[j]}_+(0)+\sigma^{[j]}_-(0)=\sigma^{(j)}(0);$$
\item $D^\zeta_{j}$ is a differential operator of order $j$ defined on a neighborhood of $\J^{-1}(\{\zeta\})\cap M_{(\h_\mathrm{top})}$ in $M_{(\h_\mathrm{top})}$, transversal to $\J^{-1}(\{\zeta\})\cap M_{(\h_\mathrm{top})}$, and for $j=0$ one has $D^\zeta_{0}=2\pi$;
\smallskip
 \item  $D_{j,F}$ (in the definite case) and $D^\pm_{j,F}$ (in the indefinite case) are $\zeta$-independent differential operators of order $2j+2-\mathrm{codim}\, F$ defined on a neighborhood of $F$ in $M$, and for the lowest index $j=\frac{1}{2}\mathrm{codim}\, F-1$ these operators equal the following constants: 
\bq\begin{split}
 D_{\mathrm{codim}\, F/2-1,F}&= 2^{\mathrm{codim}\, F/2-1} C_F,\qquad C_F=(2\pi)^2\frac{(\pi i)^{\mathrm{codim}\, F/2-1}}{|\lambda_1^F\cdots \lambda_{\mathrm{codim}\, F/2}^F|(\mathrm{codim}\, F/2-1)!},\\
 D^\pm_{\mathrm{codim}\, F/2-1,F}&= N^\pm_F C_F,\qquad\qquad\quad\; N^\pm_F\in \Z\setminus\{0\}, \label{eq:constantsDF}\end{split}
 \eq
 where  $\lambda_1^F,\ldots,\lambda_{\mathrm{codim}\, F/2}^F\in \Z\setminus \{0\}$ are the weights of the fiber-wise $T$-action on the symplectic normal bundle of $F$ in $M$, see Section \ref{subsec:4.1}, and the non-zero integers $N_F^\pm$ are explicitly determined by the invariants $n^+_F$ and  $n^-_F$, see \eqref{eq:Nvalues};
  \item $\d \M^\zeta_\mathrm{top}:=(\omega^\zeta_{\mathrm{top}})^{n-1}/(n-1)!$ and $\d F:= \omega^{\dim F/2}/(\dim F/2)!$ are the symplectic volume forms.
 \end{itemize}
 In particular, the leading order of the asymptotic expansion is given by
 \[
\ell(\zeta)=\begin{cases}1,\qquad & \M^\zeta_\mathrm{top}\neq \emptyset,\\
\inf\{\mathrm{codim}\, F/2:F\in \F,\;\J(F)=\zeta\}, & \M^\zeta_\mathrm{top}= \emptyset.
\end{cases} 
 \]
 Furthermore, the operators $D^\zeta_j$, $D_{j,F}$, and $D^\pm_{j,F}$ are natural in the following sense: if $(M',\omega',\J')$ is another Hamiltonian $T$-space and $\Phi:M\to M'$ an isomorphism of Hamiltonian $T$-spaces, then the above statements hold for $(M',\omega',\J')$ with  the operators
\[
(D^\zeta_j)':=\Phi_\ast D^\zeta_j, \qquad (D_{j,F'})':=\Phi_\ast D_{j,\Phi^{-1}(F')},  \qquad (D^\pm_{j,F'})':=\Phi_\ast D^\pm_{j,\Phi^{-1}(F')},\qquad F'\in \F',
\]
where we use the notation $\Phi_\ast D(f):=D(f\circ \Phi)$ for a differential operator $D$ defined on an open subset $U\subset M$ and $f\in \Cinft(\Phi(U))$, and $\F'=\{\Phi(F):F\in \F\}$ is the set of connected components of $M'^T$.

 \end{thm} 
 \begin{rem*}\begin{enumerate}[leftmargin=*]
 \item If $\zeta$ is a regular value of $\J$, then Theorem \ref{thm:main1} reduces to the usual asymptotic expansion \eqref{eq:regcoeff} since $(A^\zeta_j)_\mathrm{sing}=0$ and $\mathscr M^{\zeta}_\mathrm{top}=\mathscr M^{\zeta}$. 
  \item The constants  in \eqref{eq:constantsDF} are non-zero, so the singular contributions $(A^\zeta_j)_\mathrm{sing}$ do occur in general.
  \item We emphasize that the distributions $(A^\zeta_j)_\mathrm{sing}$ depend on $\zeta$ only via the condition $\J(F)=\zeta$ in the sum in \eqref{eq:distrib1}; the individual distributions $A_{j,F}$ are independent of $\zeta$. 
\item  Note that the sum over all $F\in \F$ with $F\cap \supp a\neq \emptyset$ in \eqref{eq:distrib1} is finite because $\supp a$ is compact. Moreover, as each compact subset of $M$ intersects only finitely many connected components $F\in \F$ non-trivially, one has for each $j\in \N_0$ the convergence of distributions
 \[
 (A^\zeta_j)_\mathrm{sing}=\sum_{F\in \F:\J(F)=\zeta} A_{j,F} \qquad \text{in }\;\D'(M)\otimes\S'(\t),
 \]
 where the sum on the right hand side may be infinite because we do not assume $M$ to be compact.
\item The expressions $A^\zeta_j$ in the expansion of $I^\zeta(\eps)$ are given in terms of the piecewise polynomial measures $\xi^j_\pm\in \S'(\t^\ast)$, $j \in \N_0$. This was to be expected from the Duistermaat-Heckman theorem or, more generally, from the localization principle.  But since the latter only applies to equivariantly closed differential forms, while we are considering general amplitudes, we could not rely on localization. Also notice that  the remainder in the expansion of the  generalized Witten integral does not vanish in general -- this is an exclusive phenomenon for equivariantly closed differential forms and constitutes the essence of the localization principle. In fact,  localization implies that the expansion of Theorem  \ref{thm:main1}, when applied to the original Witten integral \eqref{eq:origWitten},  consists only of finitely many terms.

 \end{enumerate}
 \end{rem*}
 Theorem \ref{thm:main1} shows that the coefficients in the asymptotic expansion are sums of two qualitatively different terms: for each $\zeta\in\t^\ast$ there are \emph{regular contributions} $(A^\zeta_j)_\mathrm{top}$ of the same form as the coefficients in  \eqref{eq:regcoeff}, and there are \emph{singular contributions} $(A^\zeta_j)_\mathrm{sing}$, which are tensor products of distributions on $M$ supported in $\J^{-1}(\{\zeta\})\cap M^T$ and some mildly exotic tempered distributions on $\t$. In particular, there are \emph{singularly leading} terms associated with each fixed point set component $F\in \F$ fulfilling $\J(F)=\zeta$, occurring at  $j=\frac{1}{2}\mathrm{codim}\, F-1$. In the latter, the obtained distribution on the manifold is simply integration over $F$, up to a constant determined uniquely by the numerical invariants $n^+_F,n^-_F$ and the weights $\lambda_1^F,\ldots,\lambda_{\mathrm{codim}\, F/2}^F$. If $\zeta$ is a regular value, the singular contributions vanish. For general singular values of $\zeta$, both regular and singular contributions appear, and in the special case that $\M^\zeta_\mathrm{top}= \emptyset$, the regular contributions vanish and the singularly leading terms actually make up the leading term of the asymptotic expansion. Let us also emphasize that all coefficient distributions in the asymptotic expansion have a clear symplectic meaning given in terms of  the symplectic structure of the strata  of $\M^\zeta$.

Note that Theorem \ref{thm:main1} gives an asymptotic expansion for each individual $\zeta\in\t^\ast$ and makes no statement about the continuity of the obtained coefficients upon variations of $\zeta$. This question is dealt with in Section \ref{sec:proof2}, where we  prove the following statement on the (dis)continuity of the coefficients, our second main result. 
\begin{thm}\label{thm:main2}
For every $\zeta_0\in \t^\ast\cong \R$, $a\in \CT(M)$, $\sigma\in \S(\t)$, and $j\in \N_0$, one has 
 \begin{multline*}
\lim_{\substack{\zeta\to \zeta_0\\ \pm(\zeta-\zeta_0)>0}} A^\zeta_j(a\otimes\sigma)=\sigma^{(j)}(0)\bigg(\int_{\mathscr M^{\zeta_0}_{\mathrm{top}}}\big<D^{\zeta_0}_{j}a\big>_T\d \mathscr M^{\zeta_0}_{\mathrm{top}}\\
+\sum_{\substack{F\in \F: \J(F)=\zeta_0,\\
\mathrm{codim}\,F/2-1\leq j,\\
F \text{ \emph{indefinite}},\\
F\cap \supp a\neq \emptyset}} \int_{F}D^\mp_{j,F} a\d F+\sum_{\substack{F\in \F: \J(F)=\zeta_0,\\
\mathrm{codim}\,F/2-1\leq j,\\
F \text{ \emph{definite}},\; s_F=\pm,\\
F\cap \supp a\neq \emptyset}} \int_{F}D_{j,F} a\d F\bigg).
 \end{multline*}
\end{thm} 
The previous theorem shows that for $j>0$ the functions $\t^\ast\owns \zeta\mapsto A^\zeta_j(a\otimes\sigma)\in \C$ are in general highly discontinuous at each singular value $\zeta=\zeta_0$, where the discontinuities are three-fold:
\begin{enumerate}[leftmargin=*]
  \item the family of definite fixed point sets $F$ contributing to the limit  depends on the sign in the limit. In particular, this produces discontinuities at $j=\mathrm{codim}\,F/2-1$ which can be quantitatively calculated in terms of explicit scalar multiples of $\int_F a \d F$ using \eqref{eq:constantsDF} and \eqref{eq:Nvalues};
 \item  the operators occurring in the contributions of the indefinite fixed point sets $F$  depend on the sign in the limit. In particular, one has $D^+_{\mathrm{codim}\,F/2-1,F}\neq D^-_{\mathrm{codim}\,F/2-1,F}$ by \eqref{eq:constantsDF} since $N_F^+\neq N_F^-$, see \eqref{eq:Nvalues}. Again, the discontinuities produced at $j=\mathrm{codim}\,F/2-1$ can be quantitatively calculated in terms of explicit scalar multiples of $\int_F a \d F$;
  \item neither when approaching $\zeta_0$ from above or below need the limit agree with the value $A^{\zeta_0}_j(a\otimes\sigma)$. This is because $A^{\zeta_0}_j$ involves the distributions $\sigma\mapsto \sigma^{[j]}_\pm(0)$, which occur in neither of the limits, and also due to the fact that the distributions on the manifold occurring in $A^{\zeta_0}_j$ are different from those appearing in the limits in Theorem \ref{thm:main2}. \interfootnotelinepenalty=1000000 
 \end{enumerate}
On the other hand, if $\zeta_0$ is a regular value of $\J$, then the result of Theorem \ref{thm:main2} reduces to the statement that the coefficients in the regular asymptotic expansion \eqref{eq:regcoeff} depend continuously on $\zeta$ at $\zeta_0$. 

\medskip

\textbf{Methods.} To overcome the problems arising in the asymptotic expansion of the generalized Witten integral at singular values of the momentum map, we do not perform a desingularization procedure, but implement a destratification process which consists of several steps. First, we linearize the phase function near each $F\in \F$ using the Guillemin-Sternberg-Marle local normal form and a  classical result by Whitney on smooth extensions of even functions defined on half-spaces.   This linearization is not the result of a monomialization of the phase function, so that  no desingularization of the critical set has taken place. As a result, for each $F \in \F$ one obtains an oscillatory integral with a clean critical set but with an integrand which is not smooth at the singular value $\J(F)$,  so that the stationary phase principle cannot be applied. Instead, in a second step  we take the Fourier transform on the Lie algebra and  split the integral at $\J(F)$ to obtain $C^\infty$-amplitudes. In a third step, we Taylor expand the integrand in powers of $\eps$, resulting in  a separation into singular and regular contributions and  a complete asymptotic description for the linearized integral. It is this separation of the contributions originating from the different strata that we call \emph{destratification}.  Finally, we translate the results obtained in the local model into meaningful expressions that live on the strata of the  symplectic quotient and can be patched together to get the stated global results. 

In this work we restricted ourselves to the simplest case of an $S^1$-action since  the derivation of  a complete asymptotic expansion of the Witten integral for arbitrary compact group actions or even torus actions is  considerably more involved. In fact, restricting to circle actions has the advantage that general phenomena such as the discontinuities of the asymptotic expansion at singular values due to the contributions by lower-dimensional strata are clearly visible while the computational effort is reduced to a minimum. Another simplification (which occurs for any abelian Lie group) is that we do not have to distinguish between orbit reduction and point reduction.

\medskip

\textbf{Previous results.} For  compact Hamiltonian $G$-manifolds $M$ arising in geometric invariant theory,  an asymptotic expansion of the Witten integral  was derived  by Jeffrey, Kiem, Kirwan, and Woolf  \cite[Theorem 34]{JKKW03} using Parseval's formula on the Lie algebra $\g$ and the localization formula of Berline-Vergne and Atiyah-Bott on the manifold $M$. The latter amounts to performing an exact stationary phase analysis only in the manifold variables,  the critical set in question being clean.  The terms in their expansion are given in terms of piecewise polynomial functions evaluated on a Gaussian. Our approach could be regarded as a singular stationary phase analysis  performed simultaneously in the  manifold \emph{and} Lie algebra variables.  In particular, the vanishing of the Lie algebra derivatives enforces a localization on level sets of the momentum map, which in the case of circle actions  leads to a precise description of the coefficients of the piecewise polynomial functions in the expansion \cite[Theorem 34]{JKKW03}  in terms  of integrals on the symplectic strata of the reduced space. In case that $0$ is a regular value of the momentum map, a stationary phase expansion similar to ours  was given for arbitrary $G$   by Meinrenken  \cite[Theorem 3.1]{meinrenken96}, generalizing a corresponding  formula of Jeffrey-Kirwan \cite[Proposition 8.10]{jeffrey-kirwan95}. In case that $0$ is a singular value, the main term in the Witten expansion  was implicitly characterized  in \cite[Theorem 18]{JKKW03}  as an integral over a desingularization of the symplectic quotient, as well as  by Lerman and Tolman \cite[Theorem 5.1]{lerman-tolman00} in the special case of $S^1$-actions. As explained above, our approach is not based on a desingularization but a destratification process, which results in an intrinsic symplectic characterization of all coefficients that could not be obtained before via desingularization techniques.

For non-compact Hamiltonian $G$-manifolds, the generalized Witten integral was studied by the second author in \cite{ramacher13} in the special case that $M=T^\ast N$ is the cotangent bundle of a smooth manifold $N$, equipped with the canonical symplectic form, and the action of $G$ on $M$ is the lift of a smooth $G$-action on $N$. Performing a stationary phase analysis in the Lie algebra and manifold variables and desingularizing partially, the leading term of the asymptotic expansion was characterized    as an integral on the top stratum of the reduced space, together with a remainder estimate. Nevertheless, the employed desingularization techniques would require further development to consider higher order terms; in particular,  the singular contributions $(A^\zeta_j)_\mathrm{sing}$  occurring in our asymptotic expansion in Theorem \ref{thm:main1} were not seen in \cite{ramacher13}. Before, Prato and Wu \cite{prato-wu94} proved a Duistermaat-Heckman type formula in the non-compact setting under a suitable properness assumption on the momentum map.

\medskip

\textbf{Applications and outlook.} The asymptotic behavior of the Witten integral plays an essential role in the derivation of residue formulas in equivariant cohomology, which was the main theme of  \cite{jeffrey-kirwan95, JKKW03} and \cite{ramacher13}. These formulas, in principle,  allow to compute the (intersection) cohomology of  reduced spaces  in terms of the equivariant cohomology of the underlying Hamiltonian space and  the fixed-point data of the group action.  The motivation for the present work was to extend the results obtained in \cite{ramacher13} for basic differential forms to arbitrary equivariant differential forms. This requires a complete asymptotic expansion of the Witten integral  and not just a computation of the leading term with a remainder estimate, which was  sufficient to deal with basic forms. 

  Residue formulas were also applied by Jeffrey and Kirwan  \cite{jeffrey-kirwan97} to Riemann-Roch numbers and the Guillemin-Sternberg conjecture under the assumption  that $0$ is a regular value of the momentum map.  
  This conjecture was proved  by Meinrenken  \cite{meinrenken96} under similar assumptions relying directly on the stationary phase expansion of the Witten integral.  In the case when $0$ is a singular value, the Guillemin-Sternberg conjecture was first proved by Meinrenken and Sjamaar \cite{meinrenken-sjamaar}. Another proof was given by Paradan and Vergne  \cite{paradan-vergne19}, both approaches 
 being based on desingularizations of the reduced space.
 
 In forthcoming papers, we intend to apply  the results derived in this paper to the study of the cohomology of symplectic quotients for general Hamiltonian circle actions via residue formulas, complementing  the previous works \cite{lerman-tolman00} and \cite{JKKW03}. In particular, we plan to extend their work to non-compact manifolds and to interpret the residues in terms of the  symplectic data of the strata of the reduced space. Our destratification approach  should also yield new insights into the Guillemin-Sternberg conjecture for circle actions.

\medskip

\textbf{Structure of the paper.}  This paper is structured as follows: Section \ref{sec:2} contains a brief introduction to Hamiltonian actions and reduced spaces, followed by the definition of the Witten integral, and gives normal forms for the momentum map and the relevant local integrals. In Section \ref{sec:19.3.2018}, complete asymptotic expansions are derived for the local integrals  via a destratification process. The coefficients in the local expansions are  interpreted geometrically in Section \ref{sec:4}. Our main results are then proved in Section \ref{sec:5}. A notation index can be found at the end of this paper.

\medskip

\textbf{Acknowledgments.} We would like to thank Panagiotis Konstantis for his  interest in our work and many stimulating  conversations. Furthermore, we warmly thank Mich\`{e}le Vergne for  many helpful and encouraging remarks, and we are grateful to Iosef Pinelis for suggesting the simplification \eqref{eq:simplif}.  Also, we would like to thank the referee for carefully reading the paper and the constructive report. The first author has received funding from the European Research Council (ERC) under the European Union's Horizon 2020 research and innovation programme (grant agreement No.\  725967).

\section{Background and setup}
\label{sec:2}

We begin by introducing some concepts in the general setting of a Hamiltonian action of a general compact connected Lie group $G$, and  then specialize to the circle case $G=T=S^1$.

\subsection{Hamiltonian actions and reduced spaces}
\label{sec:2.1}
Let $ M $ be a $2n$-dimensional 
symplectic manifold with symplectic form $\omega$. Assume that $M$ carries a Hamiltonian action of a compact connected Lie group $G$ of dimension $d$ with Lie algebra $\g$, and
denote the corresponding Kostant-Souriau momentum map  by
\bqn
\J: M \to \g^\ast,  \quad \J(p)(X)=J(X)(p),
\eqn
which is characterized by the property \bq
dJ(X) =\iota_{\widetilde X} \omega\qquad \forall\;X \in \g,\label{eq:sign24}
\eq
where $\widetilde X$ denotes the fundamental vector field on $M$ associated to $X$, $d$ is the de Rham differential, and $\iota$ denotes contraction.
Note that $\J$ is $G$-equivariant in the sense that $\J(k^{-1} p) = \Ad^\ast(k) \J(p)$. 
Let $(\Omega^\ast_G(M)_c,D)$ be the complex of compactly supported equivariant differential forms on $M$. The elements in $\Omega^\ast_G(M)_c$ can be regarded as $G$-equivariant polynomial maps $\g\to \Omega^\ast_c(M)$,  where $G$ acts on $\g$ by the adjoint action $\Ad(G)$ and on the algebra $\Omega^\ast_c(M)$ of compactly supported differential forms by the pullbacks associated to the $G$-action on $M$. The differential $D$ is then defined by 
$$D(\alpha)(X):=d(\alpha(X))+\iota_{\widetilde X}(\alpha(X)),\qquad \alpha \in \Omega^\ast_G(M)_c.$$
  We denote the cohomology of the complex $(\Omega^\ast_G(M)_c,D)$, which is called the \emph{equivariant cohomology of $M$},  by $H^\ast_G(M)_c$. Further, let
\bq
\overline \omega:=\omega-\J \label{eq:equivext}
\eq
be the equivariantly closed extension $\overline \omega$ of the symplectic form $\omega$. The approach used here is usually called the \emph{Cartan model}.

\begin{rem}[Sign convention]\label{rem:conv1} The sign convention in the definition of $D$ (and hence $\overline  \omega$) varies in the literature. We define  $D$ in coherence with \cite{atiyah-bott84}, while in  \cite{jeffrey-kirwan95} one has $D(\alpha)(X):=d(\alpha(X))-\iota_{\widetilde X}(\alpha(X))$, which leads to $\overline \omega=\omega+\J$ as opposed to our definition \eqref{eq:equivext}.
\end{rem}

From the definition of the momentum map it is clear that the kernel of its derivative is given by
\bq
\label{eq:25.12.2014}
\ker d\J|_{p}=(\g\cdot p)^\omega, \qquad p \in M,
\eq
where we denoted the symplectic complement of a subspace $V\subset T_p M$ by $V^\omega$, and   $\g  \cdot p:= \{\widetilde X_p: X \in \g\}$. Consequently, if  $\zeta\in \J(M)$ is a regular value of $\J$, the level set $\J^{-1}(\{\zeta\})$ is a {(not necessarily connected)} manifold of codimension $1$, and $T_p (\J^{-1}(\{\zeta\}))=\ker \d \J|_{p}=(\g\cdot p)^\omega$, which is equivalent to
\bqn
\widetilde X_p \not=0 \qquad \forall\; p \in \J^{-1}(\{\zeta\}), \, 0\not= X \in \g,
\eqn
compare \cite[Chapter 8]{mumford-fogarty-kirwan}. The latter condition means that all stabilizers of points $p \in \J^{-1}(\{\zeta\})$ are finite, and therefore either of exceptional or principal type, so that $\J^{-1}(\{\zeta\})/G^\zeta$ is an orbifold.  In addition, in view of the exact sequence
\bqn
0 \longrightarrow T_p(\J^{-1}(\{\zeta\})) \stackrel{d\iota^\zeta}\longrightarrow T_p M \stackrel{d \J}\longrightarrow T_\zeta \g^\ast \longrightarrow 0, \qquad p \in \J^{-1}(\{\zeta\}),
\eqn
where $\iota^\zeta: \J^{-1}(\{\zeta\}) \hookrightarrow M$ denotes the inclusion, and the corresponding dual sequence, $\J^{-1}(\{\zeta\})$ is  orientable because $M$ is orientable, compare \cite[Chapter XV.6]{lang}.  

If $\zeta$ is not a regular value,  both $\J^{-1}(\{\zeta\})$ and $\J^{-1}(\{\zeta\})/G^\zeta$  are stratified spaces. While usually  the  orbit type stratification \cite{lerman-sjamaar} is more common, for our purposes it will be more convenient to consider the infinitesimal orbit stratification  
$$\J^{-1}(\{\zeta\})=\bigcup_{(\mathfrak{h})}\J^{-1}(\{\zeta\})_{(\mathfrak{h})},  $$
see \cite[Section 3]{meinrenken-sjamaar}. Its strata consist of infinitesimal orbit type orbifolds, where an infinitesimal orbit type $(\mathfrak{h})$ of the $G$-action is an equivalence class of isotropy algebras $\mathfrak{h}\subset \g$, and  two such algebras $\mathfrak{h},\mathfrak{h}'$ are equivalent if there is an element $g\in G^\zeta$ with  $\mathfrak{h}=\mathrm{Ad}(g)\mathfrak{h}'$. 

Let us now restrict to the case   $G=T=S^1$. The infinitesimal orbit type stratification is then quite simple. In fact,  the quotient $\M^\zeta=\J^{-1}(\{\zeta\})/T$ is stratified by infinitesimal orbit types according to
 \bq
 \label{eq:21.9.2019}
\M^\zeta=  \M^\zeta_\mathrm{top} \sqcup  \M^\zeta_\mathrm{sing}, \qquad \M^\zeta_\aleph:=(\J^{-1}(\mklm{\zeta}) \cap M_{(\h_\aleph)})/T,
\eq
 where $M_{(\h_\aleph)}$ denotes the stratum of $M$ of infinitesimal orbit type $(\h_\aleph)$ with $\h_\mathrm{top}=\mklm{0}$ and  $\h_\mathrm{sing}=\t$. 
 \begin{rem}\label{rem:Mtopempty}It can happen that $\M^\zeta_\mathrm{top}$ is empty. For example, if $M=\R^2$ with  $S^1$ acting by rotations around the origin, the zero level set $\J^{-1}(\{0\})$ consists only of the origin, which is a fixed point.
\end{rem}
Since $\omega$ is non-degenerate, we see from  \eqref{eq:25.12.2014} that
\bqn
p \in M^T=M_{(\h_\mathrm{sing})}  \quad \Longleftrightarrow \quad d\J|_p=0 . 
 \eqn
Since $\J$ is constant on each $F$ we have 

\begin{lem}
\label{lem:singval}
The momentum map $\J:M \to \t^\ast$ has no critical points in $M_{(\h_\mathrm{top})}$ and its singular values are  $\mklm{\J(F): F \in \F}$.\qed
\end{lem}
As already mentioned in the introduction, the \emph{top stratum} $\M^\zeta_\mathrm{top}$  is dense in $\M^\zeta$ if non-empty and an orbifold, while  $\M^\zeta_\mathrm{sing}$ is a smooth manifold, each of its components  being identical to some $F\in \F$. Furthermore, each $F$ is a symplectic submanifold of $M$, and there is also a natural symplectic form  $\omega^\zeta_\mathrm{top}$ on $\M^\zeta_\mathrm{top}$, as explained on page \pageref{eq:mstrat}.

\subsection{The Witten integral}\label{subsec:2.2}

The central objects of our study are generalized Witten integrals of the form \eqref{eq:2}, and our main tools for their investigation will be Fourier analysis and singular stationary phase expansion. Fix an  identification  $\t\cong\t^\ast\cong\R$ and let $dx$ and $d\zeta$ be measures on $\t$  and  $\t^\ast$ that correspond to Lebesgue measure by the fixed identifications, respectively. Denote by
$$\F_\t: \S(\t^\ast) \rightarrow \S(\t), \qquad \F_\t:\S'(\t) \rightarrow \S'(\t^\ast)
$$ the Fourier transform on the Schwartz space and the space of tempered distributions, given by\footnote{Regarding  normalization conventions, see  \cite[footnotes on p.\ 125]{JeffreyKirwan98}.}
\bq
\widehat \psi(x) := (\F_\t  \psi) (x) := \int_{\t^\ast} e^{-i\eklm{\zeta,x}} \psi(\zeta) \d \zeta,\qquad\eklm{\zeta,x}:=\zeta(x),\quad x\in\t,\qquad \psi\in \S(\t^\ast),\label{eq:Fouriertransform}
\eq
and recall that $\overline \omega:=\omega-\J$.  Consider now the \emph{generalized Duistermaat-Heckman integral} 
\bq
L_\varrho:\t\to \C, \qquad L_\varrho(x):= \int_{M}e^{-i\,\overline \omega(x)} \varrho(x), 
\qquad  \varrho \in \Omega^\ast_T(M)_c, 
\label{eq:Lrhodef}
\eq
regarded as a tempered distribution in $\S'(\t)$,  compare  \cite{witten92, jeffrey-kirwan95, duistermaat94}\footnote{Jeffrey and Kirwan use the notation $\Pi_\ast(\varrho e^{-i\overline \omega})$ for our map $L_\varrho$, see \cite[p.\ 299]{jeffrey-kirwan95}.}. If $M$ is compact and $\varrho=1$, $L_\varrho$ is the classical  Duistermaat-Heckman integral,  whose $\t$-Fourier transform  is given by the pushforward $\J_\ast(\omega^n/n!)$ of the Liouville form along $\J$, which is a piecewise polynomial measure on $\t^\ast$ \cite{duistermaat-heckman82}. Motivated by this, we shall examine the behavior of the Fourier transform of $L_\varrho$ near the origin, and for this sake consider an approximation of the Dirac $\delta$-distribution centered at $\zeta\in \t^\ast$ given by 
\bqn 
\phi^\zeta_\eps(\xi):= \phi((\xi-\zeta)/\eps ) /\eps, \qquad \eps >0, 
\eqn
where  $\phi \in \CT(\t^\ast)$ is a test function satisfying $\widehat \phi(0)=1$. We are then interested in the limit 
\begin{align}\begin{split}
\label{eq:50}
\lim_{\eps\to 0^+} \eklm{ \widehat{L_\varrho},\phi^\zeta_\eps}&=
\lim_{\eps\to 0^+} \int_{\t} L_\varrho(x) \widehat {\phi^\zeta_\eps}(x) \d x=\lim_{\eps\to 0^+}\int_{\t} L_\varrho(x) e^{-i\zeta(x)}\widehat \phi(\eps x) \d x\\
&=\lim_{\eps\to 0^+}\eps^{-1}\int_{\t} \left[\int_{M} e^{i(J-\zeta)(x)/\eps} e^{-i\omega} \varrho(x/\eps) \right]  \widehat \phi(x) \d x,
\end{split}
\end{align}
which does not need to exist a priori in general, and its dependence on  $\zeta$ in a neighbourhood of $0\in \t^\ast$.  Thus, we are led to  the definition of the    \emph{Witten integral}
\bq
\label{eq:mainint}
\W^{\zeta}_{\varrho,\phi}(\eps):=\int_{\t} \left[\int_{M} e^{i(J-\zeta)(x)} e^{-i\omega} \varrho(x) \right]  \hat \phi(\eps x) \d x, \qquad \varrho \in \Omega_{T}^\ast(M)_c, \, \phi \in \S(\t^\ast),\; \eps>0,  \zeta \in \t^\ast, 
\eq
and to the investigation of its asymptotic behavior as $\eps \to 0^+$ and $\zeta \to 0$.  Note that in  this notation, $
\langle \widehat{L_\varrho}, \phi^\zeta_\eps\rangle= \W^{\zeta}_{\varrho,\phi}(\eps)$. Furthermore,  if $\varrho$ is equivariantly closed, $\W^{\zeta}_{\varrho,\phi}(\eps)$ actually only depends on the cohomology class of $\varrho$ in view of \cite[Lemma 1]{ramacher13}.

\begin{rem}\label{rem:witten} The original Witten  integral considered in  \cite{witten92} reads in our setting 
\bq
\frac{1}{(2\pi)^2 i}\int_\t\bigg[\int_M\big(e^{-i\bar  \omega}\varrho\big)(x)\bigg]e^{-\nu\frac{x^2}{2}}\d x,\qquad \nu>0,\; \varrho\in \Omega^\ast_T(M)_c. \label{eq:origWitten}
\eq
Writing $\eps:=\sqrt{\nu}$ we see that this  equals $( (2\pi)^2 i)^{-1}$ times  $\W^{\zeta}_{\varrho,\phi}(\eps)$ with $\hat \phi(x)=e^{-\frac{x^2}{2}}$ and $\zeta=0$. 
\end{rem}
To formulate \eqref{eq:mainint} more explicitly, write $\varrho$ as a finite linear combination 
\bq
\label{eq:31.7.a}
\varrho(x)=\sum_{k,m}\varrho_{k,m} x^k, \qquad \varrho_{k,m} \in \Omega^m(M)_c, \qquad k,m \in \N\cup \{0\}.
\eq
 For those $\varrho_{k,m}$ which are differential forms of odd degree, there is no appropriate power $N\in \N\cup \{0\}$ such that $\omega^N\wedge \varrho_{k,m}$ is a volume form, therefore only the $\varrho_{k,m}$ with $m$ even contribute to $\W^{\zeta}_{\varrho,\phi}(\eps)$. Thus, 
 \bqn
\W^{\zeta}_{\varrho,\phi}(\eps)=\sum_{k,m,\; m \,\text{even} }\varepsilon^{-k-1}{\int_{\t} \left[\int_{M} e^{i(J-\zeta)(x)/\eps} \frac{(-i\omega)^{n-m/2}\varrho_{k,m}}{(n-m/2)!} \right]  x^k \widehat \phi (x) \d x}.
\eqn
We associate to each $\varrho_{k,m}$ a $T$-invariant function $a_{k,m}\in \CT(M)$ by the relation
\bq
\frac{(-i\omega)^{n-m/2}\varrho_{k,m}}{(n-m/2)!}=a_{k,m}\d M,\label{eq:ampl5239580}
\eq
where $dM:={\omega^n}/{n!}$ is the symplectic volume form on $M$.  In this way, we are reduced to studying the asymptotic behavior of the \emph{generalized Witten integrals}
\bq
I_{a,\sigma}^\zeta (\varepsilon)=\int_{\t} \int_{M} e^{i\psi^\zeta(p,x)/\eps} a(p) \d M(p)\sigma(x)\d x,\qquad \zeta\in \t^\ast,\quad \varepsilon\to 0^+,\label{int}
\eq
 with amplitudes  $a\in \CT(M)$, $\sigma \in \S(\t)$, 
where the phase function $\psi^\zeta\in \Cinft(M\times\t)$ is given by
\bq
\psi^\zeta(p,x):=\J(p)(x)-\zeta(x).\label{eq:phase}
\eq
 Now, when  trying to describe the asymptotic behavior of the integral $I_{a,\sigma}^\zeta(\eps)$ by means of the generalized stationary phase principle, one faces the serious difficulty that the critical set of the phase function $\psi^\zeta$  is in general not smooth.  Indeed, due to the linear dependence of $J(x)$ on $x$ we obtain
\bqn
\gd_{x} \psi^\zeta (p,x) =\J(p)-\zeta,
\eqn
and because of  the non-degeneracy of $\omega$,
\bqn
dJ(x)=\iota_{\widetilde x} \omega=0 \quad \Longleftrightarrow \quad \widetilde x =0,
\eqn
where  $\widetilde x$ is the fundamental vector field on $M$ associated to $x$. Hence,  the critical set reads
 \begin{align}
 \label{eq:4}
 \Crit(\psi^\zeta)&:=\mklm{ (p,x) \in {M} \times \t: d\psi^{\zeta}  (p,x) =0}=\mklm{(p,x)\in \J^{-1}(\{\zeta\}) \times \t: \widetilde x_p=0 }.
\end{align}
Let us first assume that $\zeta$ is a regular value. As discussed in Section \ref{sec:2.1}, $\J^{-1}(\{\zeta\})$ is an orientable manifold,  and all stabilizers of points in $\J^{-1}(\{\zeta\})$ are finite. Consequently, $ \Crit(\psi^\zeta)=\J^{-1}(\{\zeta\}) \times \mklm{0}$; in particular, it is an orientable manifold. Even further,   the critical set of the phase function $\psi^\zeta$ is \emph{clean}  \cite[Proof of Proposition 2]{ramacher13}, which means that the transversal Hessian is non-degenerate at all points in $\Crit(\psi^\zeta)$, and the generalized stationary phase theorem \cite[Theorem C]{ramacher13} can be applied, yielding  a complete asymptotic expansion for $I_{a,\sigma}^\zeta(\eps)$.  More precisely and generally, we have the following:
\begin{prop}
\label{prop:asymptreg}
For each $\zeta\in \t^\ast$, there is a family $\{\mathscr D^\zeta_{j}\}_{j\in \N_0}$ of differential operators $\mathscr D^\zeta_{j}$ of order $j$ defined on a neighborhood of the submanifold $\J^{-1}(\{\zeta\})\cap M_{(\h_\mathrm{top})}$ which are transversal to it such that the following holds: For each  $J_0\in \N_0$ and each compact set $K\subset M$ with $\J^{-1}(\{\zeta\})\cap K$ containing only regular points of $\J$, there exists an $N_{J_0,K,\zeta}\in \N_0$ and a family of differential operators $\{D_{J_0,K,\zeta}^l\}_{0\leq l \leq N_{J_0,K,\zeta}}$ on $M$ such that for all $a\in \CT(M)$ with $\supp a\subset K$ and all $\sigma\in \S(\R)$ one has
\bq
\Big |I_{a,\sigma}^\zeta(\eps) -  \eps \sum_{j=0}^{J_0} \eps^j\sigma^{(j)}(0) \I^\zeta_{j}(a) \Big | \leq \sum_{k,l=0}^{N_{J_0,K,\zeta}}\big\Vert D_{J_0,K,\zeta}^l a\big\Vert_{\infty}\big\Vert\sigma^{(k)}\big\Vert_{\infty} \; \eps^{J_0+2}\quad \forall\; \eps>0\label{eq:estimregular}
\eq
with distributions $\I^\zeta_{j}\in\D'(M)$ of the form
\bq
\I^\zeta_{j}(a)=\int_{\mathscr M^\zeta_\mathrm{top}}\big<\mathscr D^\zeta_{j} a\big>_T\d \mathscr M^\zeta_\mathrm{top},\qquad \I^\zeta_{0}(a)=2\pi\int_{\mathscr M^\zeta_\mathrm{top}}\eklm{a}_T\d \mathscr M^\zeta_\mathrm{top},\label{eq:coeffreg}
\eq
where $\d \mathscr M^\zeta_\mathrm{top}$ is the symplectic volume form on $\mathscr M^\zeta_\mathrm{top}$, and  
for a function $f$ on $M$ and a $T$-orbit $T\cdot p\subset M$, we put  $\eklm{f}_T(T\cdot p):=\intop_{T}f(g\cdot p)\d g$, where $d g$ is the Haar measure on $T$ fixed by our identification $\t\cong\R$. 
Moreover, if $\mathscr{V}\subset \t^\ast$ is an open set such that $\J^{-1}(\mathscr{V})\cap K$ contains only regular points of $\J$, then the functions $\mathscr{V}\ni\zeta\mapsto \I^\zeta_{j,k}(a)\in \C$ are smooth for all $a\in \CT(M)$ with $\supp a\subset K$.
\end{prop}
\begin{proof}
It suffices to apply \cite[Proposition 2 and Proposition 7]{ramacher13}, where the form of the coefficients as stated here follows from \cite[Eqs.\ (18), (62)]{ramacher13}, taking into account that the function $H$ appearing in \cite[Eq.\ (62)]{ramacher13} is linear with respect to the $\t$-variable in our case, exactly as in \cite[proof of Proposition 3]{ramacher13}. 
\end{proof}

If $\zeta$ is not a regular value of $\J$, there are compact subsets $K\subset M$ such that $\Crit(\psi^\zeta)\cap K\times \t$ is not clean and the usual stationary phase theorem cannot be applied.  Instead,  we shall linearize the phase function $\psi^\zeta$ in suitable local coordinates to  derive an asymptotic expansion of the generalized Witten integral by a careful direct analysis.

\subsection{Local normal forms for the momentum map and the Witten integral}
\label{subsec:4.1}

We  shall now introduce suitable coordinates on $M$ near  the set of fixed-points 
$$M^T:=\mklm{p \in M: t \cdot p =p \; \forall \,  t \in T}.$$
 The  connected components  of $M^T$ are symplectic submanifolds of $M$ of possibly different dimensions. Recall that we denote the set of these components by  $\F$.  Let  $F \in \F$ and consider  the symplectic normal bundle $E_F:=TF^\omega\subset TM $ of $F$ in $M$. Since $F$ is symplectic, one has $TM|_{F}=TF\oplus E_F$ and $E_F$ carries a symplectic structure. In particular, the total space of $E_F$ becomes a symplectic manifold.  Furthermore, the group $T=S^1$ acts on $E_F$ fiberwise, and we may choose an $S^1$-invariant complex structure on $E_F$ compatible with the symplectic one. Each fiber of the so complexified bundle $E_F$ then splits into a direct sum of complex $1$-dimensional representations of $S^1$, so that with $ \dim F=2n_F$
\bq
E_F=\bigoplus_{j=1}^{n-n_F}\E^F_j, \label{eq:splitting111}
\eq
the $\E^F_j$ being complex line bundles over $F$.  The Lie algebra   $\t$ acts on them by
\bqn 
(\E_j^F)_p \ni v \mapsto i \lambda^F_j(x) v \in (\E_j^F)_p, \qquad   p \in F, \, x \in \t, \, \lambda^F_j\in \t^\ast\cong \R,
\eqn
where $\lambda^F_1,\ldots,\lambda^F_{n-n_F} \in \Z\setminus \{0\}$ are the weights of the $T$-action on $(E_F)_p$. They do not depend on the point $p\in F$ because $F$ is connected, and they can be grouped into positive weights   $\lambda^F_1,\dots, \lambda^F_{\ell_F^+}$ and negative weights $\lambda^F_{\ell_F^++1},\dots, \lambda^F_{\ell_F^++\ell_F^-}$. The codimension of $F$  in $M$ is given by $\mathrm{codim}\,  F=2(n-n_F)=2(\ell_F^++\ell_F^-)$.  We shall now make use of the  local normal form theorem for the momentum map $\J$ due to  Guillemin-Sternberg \cite{guillemin-sternberg84} and Marle
\cite{marle85}, 
which in our situation reads as follows:

\begin{prop}\label{prop:localnormform}
For each component $F \in \F$, there exist 
\begin{enumerate}[leftmargin=*]
\item a faithful unitary representation $\rho_F: S^1 \to (S^1)^{\ell_F^++\ell_F^-} \subset U(\ell_F^+) \times U(\ell_F^-)\subset U(\ell_F^++\ell_F^-)$ with positive weights $\lambda^F_1,\dots, \lambda^F_{\ell_F^+}\in \N$ and negative weights $\lambda^F_{\ell_F^++1},\dots, \lambda^F_{\ell_F^++\ell_F^-}\in -\N$, 
\item a principal $K_F$-bundle $P_F \rightarrow F$, where $K_F$ is a subgroup of $U(\ell_F^+) \times U(\ell_F^-)$ commuting with $\rho_F(S^1)$,
\end{enumerate}
such that $$
E_F\cong P_F\times_{K_F}\C^{\ell_F^++\ell_F^-},
$$
where $P_F\times_{K_F}\C^{\ell_F^++\ell_F^-}\rightarrow F$ is the vector bundle  associated to $P_F$ by the $K_F$-action. Furthermore, there is a symplectomorphism $\Phi_F:U_F\rightarrow V_F$ from an $S^1$-invariant neighborhood $U_F$ of $F$ in $M$ onto  an $S^1$-invariant neighborhood $V_F$ of the zero section in $E_F$, which is equivariant  with respect to the $S^1$-action on $E_F\cong P_F\times_{K_F}\C^{\ell_F^++\ell_F^-}$ given by $\rho_F$, and 
\bq
\label{eq:30.06.2018}
\J \circ \Phi_F^{-1}([\wp,w])= \frac 12 \sum_{j=1}^{\ell_F^++\ell_F^-} \lambda^F_j |w_j|^2 + \J(F), \quad w =(w_1,\dots,w_{\ell_F^++\ell_F^-}),\; [\wp,w]\in P_F\times_{K_F}\C^{\ell_F^++\ell_F^-}.
\eq
In particular,  $2\ell_F^-$ and $2\ell_F^+$ are the dimensions of the negative and positive eigenspaces  of the Hessian of $\J$ at a point of $F$, respectively.
\end{prop}
\begin{proof} See \cite[Lemma 3.1]{lerman-tolman00}. \end{proof}
Note that the local normal form neighborhood $U_F$ has the property that
\bq
U_F\cap M^T=F. \label{eq:UFMT}
\eq
By shrinking the $U_F$, we shall assume that each point $p\in M$ lies in only finitely many $U_F$. We then choose a locally finite partition of unity $\mklm{\chi_\mathrm{top},\chi_F}_{F \in \F}$ on $M$ subordinate to the open cover 
\bqn
M=M_{(\h_\mathrm{top})} \cup  \bigcup_{F \in \F} U_F,
\eqn
consisting of $T$-invariant functions such that $\chi_F\equiv 1$ in a neighborhood of $F$ for each $F\in \F$. 

The generalized Witten integral \eqref{int} with  parameter $\zeta\in \t^\ast$ and amplitudes $a\in \CT(M)$, $\sigma\in \S(\t)$ can now be written as
\bq
\label{eq:25.5.2019}
  I_{a,\sigma}^\zeta(\varepsilon)  =I_{a\chi_\mathrm{top},\sigma}^\zeta (\varepsilon)+  \sum_{F\in \mathcal \F: F\cap \supp a \neq \emptyset}  I_{a\chi_F,\sigma}^\zeta (\varepsilon),
 \eq
 which is a finite sum because $\supp a$ is compact and our partition of unity is locally finite. More abstractly and conveniently, we can write the decomposition in terms of distributions as
\bq
I^\zeta(\varepsilon)  =I_{\chi_\mathrm{top}}^\zeta (\varepsilon)+  \sum_{F\in \mathcal \F}  I_{\chi_F}^\zeta (\varepsilon) \qquad \text{in}\quad (\D(M)\otimes \S(\t))'.\label{eq:sumdistr}
\eq 
We shall focus our attention in the following on the localized integrals $I_{a\chi_F,\sigma}^\zeta (\eps)$. In terms of the coordinates provided by $\Phi_F$ we obtain with \eqref{eq:30.06.2018} and the notation
$$\zeta_F:=\zeta-{\J(F)}$$
for each of the localized integrals the  formula 
\begin{align}
\begin{split}
I_{a\chi_F,\sigma}^\zeta (\eps)&=\int_\t \int_{V_F} e^{i(\J\circ \Phi^{-1}_{F}([\wp,w])-\zeta)(x)/\eps} \, (a\chi_F)(\Phi^{-1}_{F}([\wp,w]))\d[\wp,w]\sigma(x) \d x\\\label{eq:Cespedes1737}
&=\int_\R \int_{P_F\times_{K_F}\C^{\ell_F^++\ell_F^-}} e^{i\frac x{2\eps}(\eklm{Q_Fw,w}- 2\zeta_F) }\;   (a\chi_F)(\Phi^{-1}_{F}([\wp,w]))  \d[\wp,w]\sigma(x) \d x,
\end{split}
\end{align}
where we identified $\t$ with $\R$, $d[\wp,w]$ denotes the symplectic form on $P_F\times_{K_F}\C^{\ell_F^++\ell_F^-}\cong E_F$, which agrees on $V_F$ with the pullback of the symplectic volume form $(\Phi_F^{-1})^\ast(dM|_{U_F})$,    
and we  introduced on $\C^{\ell_F^++\ell_F^-}$ the non-degenerate quadratic form
\begin{align}
\label{eq:22.10.19a}
\eklm{Q_Fw,w}&:=  \sum_{j=1}^{\ell_F^++\ell_F^-} \lambda^F_j |w_j|^2=\sum_{j=1}^{\ell_F^++\ell_F^-} \lambda^F_j  \big ((\Re w_j)^2 +(\Im w_j)^2\big ).
\end{align}
Since $\J\circ \Phi^{-1}_{F}([\wp,w])=\eklm{Q_Fw,w}$  depends only on $w\in \C^{\ell_F^++\ell_F^-}$,  we want to lift $I_{a\chi_F,\sigma}^\zeta$ to $P_F \times \C^{\ell_F^++\ell_F^-}$ in order to integrate independently over $w$ in \eqref{eq:Cespedes1737}.  
Thus, let 
\[
dF:= \omega^{\dim F/2}/(\dim F/2)!
\]
be the symplectic volume form on $F$ and let $d\wp$ be a smooth volume density on $P_F$. Then, since $P_F$ is a smooth fiber bundle over $F$, there is a differential form $\eta_F$ on $P_F$ such that
\bq
d\wp=|\eta_F\wedge \pi_{F}^\ast d F|,\label{eq:meas1}
\eq
where $\pi_{F}:P_F\to F$ is the fiber bundle projection (cf.\ \cite[p.\ 430]{lang}). Let us fix a preferred volume density $d\wp$ by demanding that the fiber volume $V(p):=\int_{\pi_F^{-1}(\{p\})}\eta_F$ is equal to $1$ for each $p\in F$. This can be simply achieved by normalizing some chosen $d\wp$ with the function $1/V$. Let now $dw$ be the symplectic volume form on $\C^{\ell_F^++\ell_F^-}$ with respect to the standard symplectic structure on $\C^{\ell_F^++\ell_F^-}$. Then we claim that the product measure $d\wp\,dw$ on $P_F\times\C^{\ell_F^++\ell_F^-}$ fulfills
\bq
d\wp\,dw=|\Pi_{F}^\ast\eta_F\wedge \tilde \pi_{F}^\ast d[\wp,w]|,\label{eq:meas2}
\eq
where  $\tilde\pi_{F}:  P_F\times\C^{\ell_F^++\ell_F^-} \rightarrow P_F\times_{K_F}\C^{\ell_F^++\ell_F^-}$ is the fiber bundle projection, and $\Pi_{F}: P_F\times\C^{\ell_F^++\ell_F^-}\to P_F$ is the projection onto the first factor. The relation \eqref{eq:meas2} can be proved as follows. By \eqref{eq:meas1}, we have
\bq
d\wp\,dw=|\Pi_{F}^\ast\eta_F\wedge (\pi_{F}\circ\Pi_{F})^\ast d F\wedge dw|,\label{eq:83083908320}
\eq
where we identified $dw$ with its pullback along the projection $P_F\times\C^{\ell_F^++\ell_F^-}\to \C^{\ell_F^++\ell_F^-}$ onto the second factor. Since \eqref{eq:meas2} is a pointwise relation, it suffices to establish it locally. Let therefore $p\in F$ and let $\mathcal U\subset M$ be a Darboux chart around $p$ such that $\mathcal U\cap F$ is the vanishing locus of the last $2(\ell_F^++\ell_F^-)$ coordinates in $\mathcal U$. Then the symplectic normal bundle  $E_F\cong P_F\times_{K_F}\C^{\ell_F^++\ell_F^-}$ of $F$ in $TM$ is trivial over $\mathcal U\cap F$ and with respect to this trivialization $E_F|_{\mathcal U\cap F}\cong (\mathcal U\cap F)\times \C^{\ell_F^++\ell_F^-}$ one has $d[\wp,w]|_{E_F|_{\mathcal U\cap F}}=dFdw$. This gives us on  $\tilde {\mathcal U}:=(\pi_{F}\circ\Pi_{F})^{-1}(\mathcal  U\cap F)\subset P_F\times\C^{\ell_F^++\ell_F^-}$ the relation $$(\pi_{F}\circ\Pi_{F})^\ast d F\wedge dw=\tilde \pi_F^\ast d[\wp,w],$$ proving \eqref{eq:meas2} on $\tilde {\mathcal U}$. Covering all of $P_F\times\C^{\ell_F^++\ell_F^-}$ with sets of the form $\tilde {\mathcal U}$ finally proves \eqref{eq:meas2}.  Thanks to \eqref{eq:meas2} and the fiber volume normalization $V\equiv 1$, we now  have for any continuous function $f$ on $P_F\times_{K_F} \C^{\ell_F^++\ell_F^-}$ with compact support the equality
\bq
\int_{P_F\times \C^{\ell_F^++\ell_F^-}}\tilde \pi_F^*(f)\d\wp \d w   =\int_{P_F\times_{K_F} \C^{\ell_F^++\ell_F^-}}f([\wp,w])  \d[\wp,w].
\label{eq:integrallift}
\eq
Applying this to our integral $I^\zeta_{a\chi_F,\sigma}$ yields with Fubini
\begin{align*}
I^\zeta_{a\chi_F,\sigma}(\eps)
&=\int_\R \int_{\C^{\ell_F^++\ell_F^-}} e^{i\frac x{2\eps}(\eklm{Q_Fw,w}- 2\zeta_F) }\;  \bigg [\int_{P_F} (a\chi_F)(\Phi^{-1}_{F}(\tilde \pi_F(\wp,w))) \d \wp \bigg ] \d w\, \sigma(x)  \d x \\
&=:\int_\R \int_{\R^{\mathrm{codim}\,  F}} e^{i\frac x{2\eps}(\eklm{Q_Fw,w}- 2\zeta_F) }\;  \tilde a_F(w) \d w\,\sigma(x)   \d x,
\end{align*}
where we identified  $\C^{\ell_F^++\ell_F^-}$ with $\R^{2(\ell_F^++\ell_F^-)}=\R^{\mathrm{codim}\,  F}$. With respect to this identification, denote by
\bqn
n_F^+:=2 \ell_F^+, \qquad n_F^-:=2 \ell_F^-
\eqn
 the real dimensions of the positive and negative eigenspaces of $Q_F$, and assume first  that both $n_F^+\neq 0$ and $n_F^-\neq 0$. 
 Introducing polar coordinates $w^+=(w_1,\dots, w_{n_F^+})=r \theta^+\in \R^{n_F^+}$ and $w^-=(w_{n_F^++1},\dots, w_{\mathrm{codim}\,  F})=s \theta ^-\in \R^{n_F^-}$ in these directions with radii $r,s>0$ and $\theta^\pm\in S^{n_F^\pm-1}\subset \R^{n_F^\pm}$ and substituting $(w_j,w_{j+1})\mapsto |\lambda_j^F|^{-1/2}(w_j,w_{j+1})$ for $1\leq j \leq \mathrm{codim}\,  F-1$, $j\in 2\N-1$, the integral $I_{a\chi_F,\sigma}^\zeta (\eps)$ reads
\bq
\label{eq:27.10.2018}
I_{a\chi_F,\sigma}^\zeta (\eps) =\Lambda_F^{-1}\int_{-\infty}^\infty \int_0^\infty \int_0^\infty e^{i\frac x{2\eps}(r^2-s^2- 2\zeta_F) }\;  {\alpha_F(r,s)} \; \d r \d s\,\sigma(x)  \d x,
\eq
where the Jacobian  of the substitution is given by $\Lambda_F^{-1} r^{n_F^+-1} s^{n_F^--1}$ with the constant
\bq
\Lambda_F:=\prod_{j=1}^{\mathrm{codim}\,  F/2}|\lambda_j^F|\in \N\label{eq:LambdaF}
\eq
and we put, with $\d\theta^\pm$ denoting the standard round measure on the Euclidean unit sphere $S^{n_F^\pm-1}$,
\begin{align}\label{eq:SFrs}
\begin{split}
\alpha_F(r,s)&:=r^{n_F^+-1} s^{n_F^--1} S_F(r,s),\qquad 
S_F(r,s):=\int_{S^{n_F^+-1}} \int_{S^{n_F^--1}} \tilde a_F (r  \theta^+, s \theta^-) \d\theta^+ \d \theta^-,\\
 \tilde a_F(w) &:=\int_{P_F} (a\chi_F)\bigg(\Phi^{-1}_{F}\bigg(\tilde \pi_F\bigg(\wp,\frac{w_1}{|\lambda_1^F|^{\frac{1}{2}}},\frac{w_2}{|\lambda_1^F|^{\frac{1}{2}}},\ldots,\frac{w_{\mathrm{codim}\,  F-1}}{|\lambda_{\mathrm{codim}\,  F/2}^F|^{\frac{1}{2}}},\frac{w_{\mathrm{codim}\,  F}}{|\lambda_{\mathrm{codim}\,  F/2}^F|^{\frac{1}{2}}}\bigg)\bigg)\bigg)  \d \wp. 
\end{split}
\end{align}
The function $\tilde a_F$ is a local cutoff of the original amplitude $a$ which has been transformed using the normal form symplectomorphism $\Phi_F$.   
 Note that the double spherical mean $S_F(r,s)$ is symmetric in $r$ and $s$.  If $n_F^+=0$ and $n^-_F\neq 0$ or $n_F^+\neq 0$ and $n^-_F=0$, the integral \eqref{eq:Cespedes1737} can be written as
\bq
\label{eq:27.10.2018a}
I_{a\chi_F,\sigma}^\zeta (\eps) = \Lambda_F^{-1}\int_{-\infty}^\infty  \int_0^\infty e^{i\frac x{2\eps}(\pm r^2 - 2\zeta_F) }\;  {\alpha_F(r)}  \d r  \,\sigma(x)\d x,\qquad n_F^\mp=0,
\eq
where, with $\tilde a_F$ as in \eqref{eq:SFrs} and with $\d\theta$ denoting the standard round measure on $S^{\mathrm{codim}\,  F-1}$, one has
\begin{align}\label{eq:RSa}
\begin{split}
\alpha_F(r)&:= r^{\mathrm{codim}\,  F-1} S_F(r), \qquad S_F(r):=\int_{S^{\mathrm{codim}\,  F-1}}  \tilde a_F(r  \theta) \d\theta,
\end{split}
\end{align}
and the spherical mean $S_F(r)$ is symmetric in $r$.

\section{Asymptotic expansions}
\label{sec:19.3.2018}
  As before, consider a $2n$-dimensional symplectic manifold $(M,\omega)$  carrying a Hamiltonian action of $T=S^1$ with momentum map $\J:M \rightarrow \t^\ast$. We are now ready to derive an asymptotic expansion  of the generalized Witten integral $I_{a,\sigma}^\zeta(\varepsilon)$ introduced in  \eqref{int}. For this sake, we shall use the decomposition \eqref{eq:25.5.2019} of $I_{a,\sigma}^\zeta(\varepsilon)$ into a global regular part $I_{a\chi_\mathrm{top},\sigma}^\zeta (\eps)$ and a finite sum of potentially singular localized  integrals $I_{a\chi_F,\sigma}^\zeta (\eps)$ which are singular iff $J(F)=\zeta$, as can be read off from their presentation \eqref{eq:27.10.2018}. In the following, we shall determine asymptotic expansions for each of those localized integrals, which are at the heart of our results. 

\subsection{Contribution of the top stratum} \label{sec:3.2.1}  By Lemma \ref{lem:singval}  the momentum map is regular on $M_{\h_\mathrm{top}}$. Therefore Proposition \ref{prop:asymptreg} yields a complete stationary phase expansion for  $I_{a\chi_\mathrm{top},\sigma}^\zeta (\eps)$. The coefficients ${Q}^\zeta_{j,k}(a\chi_\mathrm{top})$ do have a geometric interpretation in terms of integrals over $\mathscr M^\zeta_\mathrm{top}$ and are smooth in $\zeta$. Let us next turn to the more interesting contributions localized in the neighborhoods $U_F$.

\subsection{Contributions of the indefinite fixed point set components} \label{sec:3.3.1} 
Let us start by considering an $F\in \F$  for which $Q_F$ is indefinite, so that $n_F^+\neq 0$ and $n_F^-\neq 0$ holds, our departing point being the integrals \eqref{eq:27.10.2018}. While in a previous version of this paper we followed an approach of Brummelhuis, Paul, and Uribe \cite{BPU}, we shall now follow a simpler approach kindly pointed out to us by Mich\`{e}le  Vergne.  The starting point is the following classical result of Whitney on extensions of even functions.
\begin{lem}[{\cite[Theorem \ 1 on p.\ 159 and Remark on p.\ 160]{whitney43}}]\label{lem:Whitney}Given $n\in \N$ and $i\in \{1,\ldots, n\}$, let $f\in \Cinft(\R^n)$ be a function that is even in the $i$-th variable, that is,  one has $$f(x_1,\ldots,x_i,\ldots, x_n)=f(x_1,\ldots,-x_i,\ldots, x_n)\qquad \forall\; x=(x_1,\ldots,x_n)\in \R^n.$$ 
Then there exists a function 
$g\in \Cinft(\R^n)$ such that $$f(x)=g(x_1,\ldots,x^2_i,\ldots, x_n)\qquad \forall\;x=(x_1,\ldots,x_n)\in \R^n. $$
\end{lem}
This important result has a direct application to functions on $\R^2$ which are even in both variables: 
\begin{cor}\label{cor:Whitney}For every function $f\in \Cinft(\R^2)$ which is even in both variables, there is a function $g\in \Cinft(\R^2)$ such that $f(x,y)=g(x^2,y^2)$ for all $x,y\in \R$.
\end{cor}
\begin{proof}Given $f$, we apply Lemma \ref{lem:Whitney} with $i=1$ to get a function $\tilde h\in \Cinft(\R^2)$ with  $f(x,y)=\tilde h(x^2,y)$ for all $x,y\in \R$. The function $y\mapsto\tilde h(x,y)$ does not need to be even when $x<0$, but it suffices to put
\[
h(x,y):=\frac{1}{2}\big (\tilde h(x,y)+\tilde h(x,-y)\big),\qquad x,y\in \R, 
\]
to obtain a function $h\in \Cinft(\R^2)$ which is even in the second variable and satisfies  $h(x^2,y)=f(x,y)$ for all $x,y\in \R$. 
Applying now Lemma \ref{lem:Whitney} with $i=2$ to  $h$ gives us the desired function $g\in\Cinft(\R^2)$.  
\end{proof}
As a consequence, we can write the spherical mean $S_F(r,s)$, which is a compactly supported even function in both variables $r$ and $s$, in the form
\bq
S_F(r,s)=\mathscr{S}_F(r^2,s^2)\qquad r,s\in \R, \label{eq:keysimpl}
\eq
with a function $\mathscr{S}_F\in \CT(\R^2)$. Indeed, Corollary \ref{cor:Whitney} gives us a function $\mathscr{G}_F\in \Cinft(\R^2)$ satisfying the analog of \eqref{eq:keysimpl}; a function $\mathscr{S}_F$ as desired can then be constructed by multiplying $\mathscr{G}_F$ with an arbitrary cutoff function equal to $1$ on the compact set $\{(r,s):(\sqrt{|r|},\sqrt{|s|})\in \supp S_F\}\subset \R^2$. This reduces the study of the integrals \eqref{eq:27.10.2018} to the general study of integrals of the form
\bq
\label{eq:09.03.2020}
I^\zeta_{\mathscr{S},\sigma}(\eps):=\int_{\R}\int_{0}^\infty \int_0^\infty e^{i(r^2-s^2-\zeta)x/\eps}r^{2L^+-1}s^{2L^--1}\mathscr{S}(r^2,s^2)\d r \d s\, \sigma(x)\d x,\qquad \eps>0,\; \zeta \in \R,
\eq
where  $L^+,L^-\geq 1$ are two natural numbers  and $\mathscr{S}\in \CT(\R^2)$, $\sigma\in \S(\R)$ are functions,  $\S(\R)$ denoting the space of Schwartz functions on $\R$.  

The first crude, but central  asymptotics are obtained in the following 

\begin{prop}\label{prop:09.03.2020}
If $\pm\zeta>0$,  one has for every $M \in \N_0$ the asymptotics
\begin{align*}
I^\zeta_{\mathscr{S},\sigma}(\eps)
&=\eps\Bigg(\sum_{j=0}^{M}\eps^j\sigma^{(j)}(0)\sum_{k=0}^{L}\zeta^{k}\int_{|\zeta|}^\infty \bigg[\sum_{l=k}^{\min(k+j,L)} c_{j,k,l}\,t^{L-l}(\partial_--\partial_+)^{j-l+k}\bigg] \mathscr{S}\Big(\frac{t+\zeta}{2},\frac{t-\zeta}{2}\Big)\d t \\
&+\sum_{k=0}^L|\zeta|^{k}\sum_{j=L-k+1}^{M}\eps^j\sigma^{(j)}(0) \sum_{\max(0,j-L-1)\leq p+q\leq k+j-L-1}c^\mp_{j,k,p,q}(\mp \partial_\pm)^p(\partial_--\partial_+)^q\mathscr{S}((|\zeta|,|\zeta|)_\pm)\Bigg) \\
&+\O_{M}\bigg(\eps^{M+2}(1+|\zeta|^{L}+|\zeta|^{-M-1})\sum_{l=0}^{2(M+1)}\sum_{r=0}^{M+L+1}\norm{D_{l,M} \mathscr S}_\infty\int_\R|\hat\sigma(u)|(1+|u|)^{r}\d u\bigg),
\end{align*}
where $D_{l,M}$ is a differential operator of order $\leq M+1$ on $\R^2$. If $\zeta=0$, one has 
\begin{align*}
I^0_{\mathscr{S},\sigma}(\eps)
&=\eps\Bigg(\sum_{j=0}^{M}\eps^j\sigma^{(j)}(0)  \int_{0}^\infty\bigg[\sum_{l=0}^{\min(j,L)} c_{j,0,l}\,t^{L-l}(\partial_--\partial_+)^{j-l}\bigg] \mathscr{S}\Big(\frac{t}{2},\frac{t}{2}\Big)\d t\\
&+\sum_{j=L+1}^{M}\eps^j\sum_{p+q=j-L-1} \Big[c^+_{j,0,p,q}\sigma^{[j]}_+(0) \partial_-^p+ c^-_{j,0,p,q}\sigma^{[j]}_-(0)(-\partial_+)^p\Big](\partial_--\partial_+)^q\mathscr{S}(0,0) \Bigg)\\
&+\O_{M}\bigg(\eps^{M+2} \sum_{l=0}^{2(M+1)}\sum_{r=0}^{M+L+1}\norm{D_{l,M} \mathscr S}_\infty\int_\R|\hat\sigma(u)|(1+|u|)^{r}\d u\bigg),
\end{align*}
where $L:=L^++L^--2$, $(x,y)_+:=(x,0)$, $(x,y)_-:=(0,y)$, the expressions $\sigma^{[j]}_\pm(0)$ are as in \eqref{eq:distrib0}, $\partial_{+}\mathscr{S}$ and $\partial_-\mathscr{S}$ are the partial derivatives of $\mathscr{S}$ with respect to the first and second variable, respectively,  and $c_{j,k,l},c^\pm_{j,k,p,q}\in \C$ are explicitly computable in terms of $L^+,L^-$. Some particular values are 
\begin{align}\begin{split}
c_{0,0,0}&=2^{-2-L}\pi,\\
 c^\pm_{L+1,0,0,0}&=2^{-2-L}\pi(-i)^{L-1} \sum_{l=0}^{L}\frac{(\pm 1)^{L-l+1}}{L-l+1}\sum_{\substack{l^++l^-=l\\  0\leq l^\pm\leq L^\pm-1}}(-1)^{l^+}{L^+-1 \choose {l^+}}{L^--1 \choose {l^-}}.\label{eq:constants000}\end{split}
\end{align}
\end{prop}
To begin, notice that the critical set of the  phase function in \eqref{eq:09.03.2020}, regarded as a function on $\R^3$, becomes singular for $\zeta=0$.  One could therefore be inclined to desingularize the critical set in some way in order to be able to apply the stationary phase theorem. This was the way followed in  \cite{ramacher13}, which was sufficient to compute the leading term and an estimate for the remainder.  Nevertheless, serious difficulties arise when trying to find a complete asymptotic expansion. Instead, the proof of Proposition \ref{prop:09.03.2020} will be  based on a  \emph{destratification process}, which we shall carry out  in the following.      
As a first step, we linearize the phase function by means of the substitution $r^2=T$, $s^2=U$, yielding
\[
I^\zeta_{\mathscr{S},\sigma}(\eps)=\frac{1}{4}\int_{\R}\int_{0}^\infty \int_0^\infty e^{i(T-U-\zeta)x/\eps}T^{L^+-1}U^{L^--1}\mathscr{S}(T,U)\d T \d U\, \sigma(x)\d x.
\]
Performing the substitutions $U-T=u$, $U+T=t$ we then obtain 
with $L:=L^++L^--2$ the formula
\bqn
I^\zeta_{\mathscr{S},\sigma}(\eps)
=2^{-3-L}\int_{\R}\int_\R e^{-i(u+\zeta)x/\eps} \sigma(x) \int_{|u|}^\infty (t-u)^{L^+-1}(t+u)^{L^--1}\mathscr{S}\Big(\frac{t-u}{2},\frac{t+u}{2}\Big)\d t  \d x \d u,
\eqn
where the $t$-integrals correspond to integrals over the level sets $\mklm{(T,U) \mid U-T=u}\subset \R_+^2$.   The critical set of the linearized phase function\footnote{Note that the linearization of the phase function is not the result of a monomialization, so that  no desingularization of the critical set as in \cite{ramacher13} has taken place. In particular, as a consequence of a desingularization, an exceptional divisor would have to appear, which is not the case here.} now consists of the single point $(u,x)=(-\zeta,0)$, but a stationary phase analysis is not possible since the amplitude is not smooth at $u=0$ due to the integral limit $|u|$. Instead, we carry out the Fourier transform on the Lie algebra, and split the $u$-integral at $0$ in order to obtain smooth coefficients.  Expanding in addition the binomial expressions $(t\mp u)^{L^\pm-1}$ and substituting $u \mapsto \eps u-\zeta$, we arrive at  
\begin{align}\begin{split}
I^\zeta_{\mathscr{S},\sigma}(\eps)
&=2^{-3-L}\eps \sum_{l=0}^{L}c_{l}\Bigg[\int_{\zeta/\eps}^\infty \hat \sigma(u)(\eps u-\zeta)^{l} \int_{\eps u-\zeta}^\infty t^{L-l}\mathscr{S}\Big(\frac{t- \eps u+\zeta }{2},\frac{t+ \eps u-\zeta}{2}\Big)\d t  \d u\\
&+\int_{-\infty}^{\zeta/\eps } \hat \sigma(u) (\eps u-\zeta)^{l} \int_{-\eps u+\zeta}^\infty t^{L-l}\mathscr{S}\Big(\frac{t-\eps u+\zeta}{2},\frac{t+ \eps  u-\zeta }{2}\Big)\d t  \d u\Bigg]\label{eq:int85u3938}\end{split}
\end{align}
with
\bq
c_{l}:=\sum_{\substack{l^++l^-=l\\ 0\leq  l^\pm\leq L^\pm-1}}(-1)^{l^+}{L^+-1 \choose {l^+}}{L^--1 \choose {l^-}},\qquad l\in \N_0.\label{eq:cl}
\eq
In order to obtain an expansion in powers of $\eps$, it is natural to  Taylor expand the $t$-integral at $\eps=0$ which, by the following lemma,  will result  in  a separation into singular and regular contributions. 

\begin{lem}\label{lem:taylor}For $N\in \N_0$, $\zeta\in \R$, and $\mathscr{S}\in \CT(\R^2)$,  define two functions $F^\pm_{N,\zeta,\mathscr{S}}\in \Cinft(\R)$ by
\[
F^\pm_{N,\zeta,\mathscr{S}}(v):=\int_{\pm(v -\zeta)}^\infty t^{N}\mathscr{S}\Big(\frac{t-v+\zeta}{2},\frac{t+v-\zeta}{2}\Big)\d t.
\]
Then, for $m\in \N_0$, the $m$-th derivative of $F^\pm_{N,\zeta,\mathscr{S}}$ is of the form
\bqn\begin{split}
&(F^\pm_{N,\zeta,\mathscr{S}})^{(m)}(v)=\frac{1}{2^m} \int_{\pm(v -\zeta)}^\infty t^{N}(\partial_--\partial_+)^m \mathscr{S}\Big(\frac{t-v+\zeta}{2},\frac{t+v-\zeta}{2}\Big)\d t \\
 &+\sum_{i=0}^{m-1}(\mp 1)^{m+i}(\pm(v -\zeta))^{\max(0,N+1-m+i)}\sum_{p+q=i}C_{N,m,p,q}(\pm \partial_\mp)^p(\partial_--\partial_+)^q\mathscr{S}((\zeta-v,v-\zeta)_\mp),
\end{split}
\eqn
where the notation is as in Proposition \ref{prop:09.03.2020}, and the constants $C_{N,m,p,q}\in \R$ satisfy \bq\begin{alignedat}{2}
C_{N,m,m-1,0}&=1,\qquad  &C_{N,N+1,0,0}&=(-1)^NN!,\\
C_{N,m,0,m-1}&=2^{1-m},\qquad &C_{N,m,p,q}&= 0\quad \text{if }p+q< \max(0,m-1-N).\label{eq:CNmpq}
\end{alignedat}\eq
\end{lem}
\begin{proof}
For $m=0$ the claim is trivially true; there are no constants $C_{N,0,p,q}$ because the sum over $i$ is empty. For $m=1$ we get
\begin{align*}
(F^\pm_{N,\zeta,\mathscr{S}})^{(1)}(v)&=\frac{d}{dv}\int_{\pm(v -\zeta)}^\infty t^{N}\mathscr{S}\Big(\frac{t-v+\zeta}{2},\frac{t+v-\zeta}{2}\Big)\d t\\
&=\frac{1}{2} \int_{\pm(v -\zeta)}^\infty t^{N}(\partial_--\partial_+)\mathscr{S}\Big(\frac{t-v+\zeta}{2},\frac{t+v-\zeta}{2}\Big)\d t \mp (\pm(v -\zeta))^{N}\mathscr{S}((\zeta-v,v-\zeta)_\mp),
\end{align*}
so that the claim holds. Assuming now that it holds for some $m\geq 1$, we obtain
\begin{align*}
&(F^\pm_{N,\zeta,\mathscr{S}})^{(m+1)}(v)\\
&=\frac{d}{dv}\bigg(\frac{1}{2^m} \int_{\pm(v -\zeta)}^\infty t^{N}(\partial_--\partial_+)^m \mathscr{S}\Big(\frac{t-v+\zeta}{2},\frac{t+v-\zeta}{2}\Big)\d t \\
&\qquad+ \sum_{i=0}^{m-1}(\mp 1)^{m+i}(\pm(v -\zeta))^{\max(0,N+1-m+i)}\sum_{p+q=i}C_{N,m,p,q}(\pm \partial_\mp)^p(\partial_--\partial_+)^q\mathscr{S}((\zeta-v,v-\zeta)_\mp)\bigg)\\
&=\frac{1}{2^{m+1}} \int_{\pm(v -\zeta)}^\infty t^{N}(\partial_--\partial_+)^{m+1} \mathscr{S}\Big(\frac{t-v+\zeta}{2},\frac{t+v-\zeta}{2}\Big)\d t\\
&\qquad\mp\frac{1}{2^m} (\pm(v -\zeta))^{N}(\partial_--\partial_+)^m \mathscr{S}((\zeta-v,v-\zeta)_\mp) \\
&\qquad+\sum_{i=\max(0,m-N)}^{m-1}(\mp 1)^{m+i}(\pm 1)(N+1-m+i)(\pm(v -\zeta))^{\max(0,N+1-m+i)-1}\\
&\qquad\qquad\qquad\qquad\cdot\sum_{p+q=i}C_{N,m,p,q}(\pm \partial_\mp)^p(\partial_--\partial_+)^q\mathscr{S}((\zeta-v,v-\zeta)_\mp)
\\
&\qquad+\sum_{i=0}^{m-1}(\mp 1)^{m+i}(\pm(v -\zeta))^{\max(0,N+1-m+i)}\sum_{p+q=i}C_{N,m,p,q}(\pm \partial_\mp)^{p+1}(\partial_--\partial_+)^q\mathscr{S}((\zeta-v,v-\zeta)_\mp).
\end{align*}
Taking into account that $\max(0,N+1-m+i)\geq 1$ in the summand of $\sum_{i=\max(0,m-N)}^{m-1}$ and performing the 
substitutions $i\mapsto i+1$, $p\mapsto p-1$ in the final sums, the expression for $(F^\pm_{N,\zeta,\mathscr{S}})^{(m+1)}(v)$ becomes
\begin{align*}
&\frac{1}{2^{m+1}} \int_{\pm(v -\zeta)}^\infty t^{N}(\partial_--\partial_+)^{m+1} \mathscr{S}\Big(\frac{t-v+\zeta}{2},\frac{t+v-\zeta}{2}\Big)\d t\\
&\mp\frac{1}{2^m} (\pm(v -\zeta))^{N}(\partial_--\partial_+)^m \mathscr{S}((\zeta-v,v-\zeta)_\mp) \\
&-\sum_{i=\max(0,m-N)}^{m-1}(\mp 1)^{m+1+i}(N+1-m+i)(\pm(v -\zeta))^{\max(0,N+1-m-1+i)}\\
&\qquad\qquad\qquad\qquad\cdot\sum_{p+q=i}C_{N,m,p,q}(\pm \partial_\mp)^p(\partial_--\partial_+)^q\mathscr{S}((\zeta-v,v-\zeta)_\mp)
\\
&+\sum_{i=1}^{m}(\mp 1)^{m+1+i}(\pm(v -\zeta))^{\max(0,N+1-m-1+i)}\sum_{\substack{p+q=i\\ p\geq 1}}C_{N,m,p-1,q}(\pm \partial_\mp)^{p}(\partial_--\partial_+)^q\mathscr{S}((\zeta-v,v-\zeta)_\mp).
\end{align*}
This is of the claimed form with
\[
C_{N,m+1,p,q}=\begin{cases}2^{-m},\qquad &p=0,q=m\\
C_{N,m,p-1,q},\qquad &p+q=m,\;p\geq 1,\\
C_{N,m,p-1,q}-(N+1-m+p+q)C_{N,m,p,q},\quad &m-1\geq p+q\geq \max(0,m-N),\; p\geq 1,\\
-(N+1-m+q)C_{N,m,0,q},\quad &m-1\geq q\geq \max(0,m-N),\;p=0,\\
C_{N,m,p-1,q}, & 1 \leq p+q\leq \max(0,m-N)-1,\;p\geq 1,\\
0, & \text{else}.
\end{cases}
\]
The first two lines prove the two equations on the left in \eqref{eq:CNmpq} since $C_{N,m,m-1,q}=1$ by the induction hypothesis. In the only case above where $p+q<\max(0,m-N)$, we have $C_{N,m+1,p,q}=C_{N,m,p-1,q}$, which is zero by the induction hypothesis because $p-1+q<\max(0,m-1-N)$. Finally, inspecting the case $q=p=0$ yields $C_{N,m+1,0,0}=-(N+1-m)C_{N,m,0,0}$ for $1\leq m\leq N$, $C_{N,1,0,0}=1$, which gives us by iteration the desired formula $C_{N,N+1,0,0}=(-1)^NN!$, so that \eqref{eq:CNmpq} is fully verified.
\end{proof}
As an immediate consequence, we get
\begin{cor}\label{cor:taylorremainder} In the situation of Lemma \ref{lem:taylor},  
there are differential operators $D^\pm_{N,m}$ of order $m$ on $\R^2$ such that 
\bq
|(F^\pm_{N,\zeta,\mathscr{S}})^{(m)}(v)|\leq (1+|v-\zeta|^N)\big\Vert D^\pm_{N,m}\mathscr{S}\big\Vert_\infty\qquad \forall\;v, \zeta\in \R.
\eq\qed
\end{cor}
We are now ready to prove Proposition \ref{prop:09.03.2020}. 
\begin{proof}[Proof of Proposition \ref{prop:09.03.2020}]\label{proof:prop1}
We perform in \eqref{eq:int85u3938} for each $u$ a Taylor expansion with Lagrange remainder of the functions $\eps \to F^\pm_{N,\zeta,\mathscr{S}}(\eps u)$ at $\eps=0$, where $N=L-l$. With  Corollary \ref{cor:taylorremainder}  this yields for arbitrary Taylor cutoff orders $M^+,M^-\in \N_0$
\begin{align}
\nonumber I^\zeta_{\mathscr{S},\sigma}(\eps)
&=2^{-3-L}\eps \sum_{l=0}^{L}c_l\Bigg[\int_{\zeta/\eps }^\infty\hat\sigma(u )(\eps u-\zeta)^{l}\bigg(\sum_{m^+=0}^{M^+}\frac{(\eps u)^{m^+}(F^+_{L-l,\zeta,\mathscr{S}})^{(m^+)}(0)}{m^+!}\bigg) \d u + R^+_{\mathscr{S},\sigma,l,M^+}(\zeta,\eps)\\
&+\int_{-\infty}^{\zeta/\eps }\hat\sigma(u)(\eps u-\zeta)^{l} \bigg(\sum_{m^-=0}^{M^-}\frac{(\eps u) ^{m^-}(F^-_{L-l,\zeta,\mathscr{S}})^{(m^-)}(0)}{m^-!}\bigg) \d u + R^-_{\mathscr{S},\sigma,l,M^-}(\zeta,\eps)\Bigg],\label{eq:2819230658906}
\end{align}
where 
\begin{align*}
\nonumber |R^\pm_{\mathscr{S},\sigma,l,M^\pm}(\zeta,\eps)|&\leq \eps^{M^\pm+1}\frac{\big\Vert D^\pm_{L-l,M^\pm+1}\mathscr{S}\big\Vert_\infty}{(M^\pm+1)!}\int_{\R}|\hat\sigma(u)||\eps u-\zeta|^{l} |u|^{M^\pm+1}\big[1+\sup_{|t|\leq 1} |t\eps u-\zeta|^{L-l}\big]\d u\\
&=\O_{M^\pm}\bigg(\eps^{M^\pm+1}(1+|\zeta|^{L})\big\Vert D^\pm_{L-l,M^\pm+1}\mathscr{S}\big\Vert_\infty\sum_{r=0}^{M^\pm+L+1}\int_\R|\hat\sigma(u)|(1+|u|)^{r}\d u\bigg).
\end{align*}
Plugging in the derivatives at $0$ from Lemma \ref{lem:taylor}, choosing $M^+=M^-=M$, and expanding the binomial expression $(\eps u-\zeta)^{l}$, we arrive after some further basic manipulations at the formula
\begin{align*}
I^\zeta_{\mathscr{S},\sigma}(\eps)
&=2^{-3-L}\eps\sum_{j=0}^{M}\eps^j\sum_{l=0}^{L}(-1)^lc_l\sum_{k=0}^{\min(j,l)}{l\choose k}\frac{(-1)^{k}}{(j-k)!}\\
&\Bigg[\int_{\zeta/\eps}^\infty\hat\sigma(u)u^{j}\d u\Bigg(\frac{\zeta^{l-k}}{2^{j-k}} \int_{-\zeta}^\infty t^{L-l}(\partial_--\partial_+)^{j-k} \mathscr{S}\Big(\frac{t+\zeta}{2},\frac{t-\zeta}{2}\Big)\d t \\
&+ \sum_{m=0}^{j-k-1}(-1)^{m+1}(-\zeta)^{\max(l,L-m)-k}\sum_{p+q=j-k-m-1}C_{L-l,{j-k},p,q}\,\partial_-^p(\partial_--\partial_+)^q\mathscr{S}(0,-\zeta)\Bigg)  \\
&+\int_{-\infty}^{\zeta/\eps}\hat\sigma(u)u^{j}\d u\Bigg(\frac{\zeta^{l-k}}{2^{j-k}} \int_{\zeta}^\infty t^{L-l}(\partial_--\partial_+)^{j-k} \mathscr{S}\Big(\frac{t+\zeta}{2},\frac{t-\zeta}{2}\Big)\d t \\
&+ \sum_{m=0}^{{j-k}-1}\zeta^{\max(l,L-m)-k}\sum_{p+q=j-k-m-1}C_{L-l,{j-k},p,q}\,(-\partial_+)^p(\partial_--\partial_+)^q\mathscr{S}(\zeta,0)\Bigg) \Bigg]\\
&+\O_{M}\bigg(\eps^{M+2}(1+|\zeta|^{L})\sum_{l=0}^{2(M+1)}\sum_{r=0}^{M+L+1}\big\Vert D^l_{M}\mathscr{S}\big\Vert_\infty\int_\R|\hat\sigma(u)|(1+|u|)^{r}\d u\bigg)
\end{align*}
with a new family $\{D^l_M\}_{l}$ of differential operators defined by $D^+_{l,M+1}$ for $0\leq l\leq M+1$ and by $D^-_{l-M-1,M+1}$ for $M+2\leq l\leq 2M+2$. 
Next, we note that for $\zeta>0$ and each $j,N\in \N\cup \{0\}$ one has
\begin{align*}
\intop_{-\infty}^{\zeta/\eps}\hat\sigma(u)u^{j}\d u&=\intop_{-\infty}^{\infty}\hat\sigma(u)u^{j}\d u-\intop_{\zeta/\eps}^{\infty}\hat\sigma(u)u^{j}\d u=2\pi (-i)^{j}\sigma^{(j)}(0)-\intop_{\zeta/\eps}^{\infty}\hat\sigma(u)u^{j}\d u,\\
\bigg|\int_{\zeta/\eps}^\infty\hat\sigma(u)u^{j}\d u\bigg|&=\bigg|\int_{\zeta/\eps}^\infty u^{-N}\hat\sigma(u)u^{j+N}\d u\bigg|\leq \int_{\zeta/\eps}^\infty| u^{-N}\hat\sigma(u)u^{j+N}|\d u\leq \eps^{N}\zeta^{-N} 
\int_\R|\hat\sigma(u)u^{j+N}|\d u,
\end{align*}
and similarly for $\zeta<0$. This allows us to replace $\intop_{-\infty}^{\zeta/\eps}\hat\sigma(u)u^{j}\d u$ by $2\pi (-i)^{j}\sigma^{(j)}(0)$ up to an error estimated by arbitrarily high powers of $\eps$, at the cost of getting equally high  negative powers of $\zeta$. Together with the above estimates for the remainder of \eqref{eq:2819230658906}, we get the claimed remainder estimate.

To obtain the claimed form of the coefficients in the asymptotic expansion, we now substitute $k\mapsto l-k$, swap the sums over $k$ and $l$, restrict the range of $m$ using the vanishing relation in \eqref{eq:CNmpq}, and substitute $m\mapsto L-l-m+k$. This yields for $\pm\zeta>0$ 
\begin{align*}
&I^\zeta_{\mathscr{S},\sigma}(\eps)
=2^{-2-L}\pi\eps\sum_{j=0}^{M}\eps^j(-i)^{j}\sigma^{(j)}(0)\sum_{k=0}^{L}(-1)^k\sum_{l=k}^{\min(k+j,L)}{l\choose k}\frac{c_l}{(j-l+k)!}\\
&\cdot\Bigg(\frac{\zeta^{k}}{2^{j-l+k}} \int_{|\zeta|}^\infty t^{L-l}(\partial_--\partial_+)^{j-l+k} \mathscr{S}\Big(\frac{t+\zeta}{2},\frac{t-\zeta}{2}\Big)\d t \\
&+ \sum_{m=L-j+1}^{L-l+k}(\pm 1)^{L-l-m+k+1}  |\zeta|^{\max(k,m)}\sum_{p+q=m+j-L-1}C_{L-l,{j+k-l},p,q}\,(\mp \partial_\pm)^p(\partial_--\partial_+)^q\mathscr{S}((|\zeta|,|\zeta|)_\pm)\Bigg)
\end{align*}
up to the remainder term, and for $\zeta=0$ 
\begin{align}\begin{split}
I^0_{\mathscr{S},\sigma}(\eps)
&=2^{-2-L}\pi\eps\Bigg[\sum_{j=0}^{M}\eps^j(-i)^j\sigma^{(j)}(0) \sum_{l=0}^{\min(j,L)}\frac{c_l}{2^{j-l}(j-l)!} \int_{0}^\infty t^{L-l}(\partial_--\partial_+)^{j-l} \mathscr{S}\Big(\frac{t}{2},\frac{t}{2}\Big)\d t\\
&+\sum_{j=L+1}^{M}\eps^j (-i)^j\sum_{l=0}^{L}\frac{c_l}{(j-l)!}\Bigg(\sigma^{[j]}_-(0)\sum_{p+q=j-L-1}C_{L-l,{j-l},p,q}\,(-\partial_+)^p(\partial_--\partial_+)^q\mathscr{S}(0,0) \\
&+ (-1)^{L-l+1}\sigma^{[j]}_+(0) \sum_{p+q=j-L-1}C_{L-l,{j-l},p,q}\,\partial_-^p(\partial_--\partial_+)^q\mathscr{S}(0,0)\Bigg) \Bigg]\label{eq:3959166261901252}\end{split}
\end{align}
up to the remainder term. The formula for $\pm\zeta>0$ immediately leads us to define
\begin{align*}
c_{j,k,l}:=2^{-2-L}\pi(-i)^{j}(-1)^k{l\choose k}\frac{c_l}{2^{j-l+k}(j-l+k)!}.
\end{align*}
In particular, we find $c_{0,0,0}=2^{-2-L}\pi$ as claimed  since $c_0=1$. Similarly, $c^\pm_{j,k,p,q}$ can be computed explicitly from the formula for $\pm\zeta>0$ but in a much more complicated way than $c_{j,k,l}$ due to the presence of $\max(k,m)$ in the exponent of $|\zeta|$. However, we can read off $c^\pm_{j,0,p,q}$ from \eqref{eq:3959166261901252}, obtaining
\[
c^\pm_{j,0,p,q}:=2^{-2-L}\pi(-i)^j \sum_{l=0}^{L}(\mp 1)^{L-l+1}\frac{c_lC_{L-l,{j-l},p,q} }{(j-l)!},\qquad j\geq L+1.
\]
Finally, for the computation of $c_{L+1,0,0,0}$ we use that $C_{L-l,L+1-l,0,0}=(-1)^{L-l}(L-l)!$ by  \eqref{eq:CNmpq}.
\end{proof}

\subsection{Contributions of the definite fixed point set  components} \label{sec:3.3.2} 
It remains to study the less difficult case of a fixed point set component $F\in \F$ for which $Q_F$ is definite, so that one either has $n_F^+=\mathrm{codim} \, F$, $n_F^-=0$ or $n_F^-= \mathrm{codim} \, F$, $n_F^+=0$.  The spherical mean $S_F$ from \eqref{eq:RSa} is then an even function of only one variable that we write in the form
\bqn
S_F(r)=\mathscr{S}_F(r^2)\qquad r\in \R,
\eqn
with $\mathscr{S}_F\in \CT(\R)$ by applying again Whitney's classical result \cite{whitney43}. 
This reduces the study of the integrals \eqref{eq:27.10.2018a} to the general study of integrals of the form
\bq
\label{eq:09.03.20200}
I^{\pm,\zeta}_{\mathscr{S},\sigma}(\eps):=\int_{\R} \int_0^\infty e^{i(\pm r^2-\zeta)x/\eps}r^{2L-1}\mathscr{S}(r^2)\d r \, \sigma(x)\d x,\qquad \eps>0,\; \zeta \in \R,
\eq
where  $L\geq 1$ is a natural number  and $\mathscr{S}\in \CT(\R)$, $\sigma\in \S(\R)$ are functions. 
\begin{prop}\label{prop:4.3.def}If $\pm\zeta>0$, then one has for each $M\in \N_0$, $M\geq L-1$, the asymptotic estimates
\begin{align*}
I^{\pm,\zeta}_{\mathscr{S},\sigma}(\eps)
&=\eps\sum_{k=0}^{L-1} \zeta^k \sum_{j=L-1-k}^{M-k}\eps^{j} \sigma^{(j)}(0)(\pm 1)^{j+k}c_{j,k}\mathscr{S}^{(j+k+1-L)}(\pm \zeta)\\
&\qquad +\O_{M}\bigg(\eps^{M+2}(1+|\zeta|^{L-1}+|\zeta|^{-M-1})\sum_{r=0}^{M+L+1}\Big(\big\Vert\mathscr{S}^{(r)}\big\Vert_\infty+\int_\R|\hat\sigma(u)|(1+|u|)^{r}\d u\Big)\bigg),\\
I^{\mp,\zeta}_{\mathscr{S},\sigma}(\eps)
&=\O_{M}\bigg(\eps^{M+2}(1+|\zeta|^{L-1}+|\zeta|^{-M-1})\sum_{r=0}^{M+L+1}\Big(\big\Vert\mathscr{S}^{(r)}\big\Vert_\infty+\int_\R|\hat\sigma(u)|(1+|u|)^{r}\d u\Big)\bigg),
\end{align*}
where 
\[
c_{j,k}=\pi{L-1\choose k}\frac{i^j}{(j+k+1-L)!}.
\]
If $\zeta=0$, then one has for each $M\in \N_0$, $M\geq L-1$, the asymptotic estimates
\begin{align*}
I^{\pm,0}_{\mathscr{S},\sigma}(\eps)
&=\eps\sum_{j=L-1}^{M}\eps^{j}\sigma^{[j]}_\mp(0)(\pm 1)^j c_{j,0}\mathscr{S}^{(j+1-L)}(0)\\
&\qquad +\O_{M}\bigg(\eps^{M+2}\sum_{r=0}^{M+L+1}\Big(\big\Vert\mathscr{S}^{(r)}\big\Vert_\infty+\int_\R|\hat\sigma(u)|(1+|u|)^{r}\d u\Big)\bigg).
\end{align*}

\end{prop}\begin{proof}
After substituting $r^2=u$, $u\mapsto \eps u\pm \zeta$, we can write 
\begin{align*}
I^{\pm,\zeta}_{\mathscr{S},\sigma}(\eps)
&=\frac{\eps}{2}\sum_{j=0}^{L-1}\eps^j{L-1\choose j}(\pm \zeta)^{L-1-j} \int_{\mp \zeta/\eps}^\infty u^j\mathscr{S}(\eps u\pm \zeta)\hat\sigma(\mp u)\d u.
\end{align*}
Performing a Taylor expansion of the function $x\mapsto\mathscr{S}(x\pm \zeta)$ at $x=0$ yields, similarly as in \eqref{eq:2819230658906}, for every $M\in \N_0$
\begin{align*}
I^{\pm,\zeta}_{\mathscr{S},\sigma}(\eps)
&=\frac{\eps}{2}\sum_{m=0}^{L+M-1}\eps^{m}\int_{\mp \zeta/\eps}^\infty u^{m}\hat\sigma(\mp u)\d u\sum_{j=\max(0,m-M)}^{\min(m,L-1)}{L-1\choose j}(\pm \zeta)^{L-1-j}\frac{\mathscr{S}^{(m-j)}(\pm \zeta)}{(m-j)!}\\
&\qquad +\O_{M}\bigg(\eps^{L+M+1}(1+|\zeta|^{L-1})\sum_{r=0}^{M+L}\Big(\big\Vert\mathscr{S}^{(r)}\big\Vert_\infty+\int_\R|\hat\sigma(u)|(1+|u|)^{r}\d u\Big)\bigg).
\end{align*}
We can now finish the proof analogously to the proof of Proposition \ref{prop:09.03.2020}.
\end{proof}

\section{Geometric interpretation of the coefficients}
\label{sec:4}

We shall now interpret the coefficients obtained in the previous section geometrically for $\zeta=0$.  
\subsection{Local geometric interpretation}\label{sec:geomint}
As a first step, we specialize to the situation that in Propositions \ref{prop:09.03.2020} and \ref{prop:4.3.def} the function $\mathscr{S}=\mathscr{S}_f$ is given in terms of a spherical mean value of an arbitrary function $f\in \CT(\R^{\mathrm{codim} \, F})$. 
Turning first to the indefinite case we thus write
\bq
\label{eq:doublesphermean}
 S_f(r,s):=\int_{S^{n_F^+-1}} \int_{S^{n_F^--1}}   f (r  \theta^+, s \theta^-) \d\theta^+ \d \theta^-=:\mathscr{S}_f(r^2,s^2),\qquad r,s\in \R.
\eq
In order to interpret the coefficients  in Proposition \ref{prop:09.03.2020} for such functions $\mathscr{S}=\mathscr{S}_f$, we first observe the following fundamental relation between derivatives of $S_f$ and $\mathscr{S}_f$. For each $k\in \N_0$, one has
\begin{align}\begin{split}
\partial_+^k\mathscr{S}_f(t,u)&=\delta_r^k S_f(\sqrt{t},\sqrt{u})\qquad \forall\; (t,u)\in (0,\infty)\times [0,\infty),\\
 \partial_-^k\mathscr{S}_f(t,u)&=\delta_s^k S_f(\sqrt{t},\sqrt{u})\qquad \forall\; (t,u)\in [0,\infty)\times (0,\infty),\label{eq:derivrel}\end{split}
\end{align}
where, as in Proposition \ref{prop:09.03.2020}, $\partial_+\mathscr{S}_f$ and $\partial_-\mathscr{S}$ are the partial derivatives of $\mathscr{S}_f$ with respect to the first and second variable, respectively, and $\delta_r^k=(\delta_r)^k,$ $\delta_s^k=(\delta_s)^k$ are powers of the operators
\[
\delta_r S_f(r,s):=\frac{1}{2r}\frac{\partial S_f}{\partial r}(r,s)\quad \forall\; r>0,s\geq 0,\qquad\qquad \delta_s S_f(r,s):=\frac{1}{2s}\frac{\partial S_f}{\partial s}(r,s)\quad \forall\; r\geq 0,s>0.
\]
\begin{definition}For $\zeta \in \R$ we introduce the full, pointed, and slit quadrics
\begin{align*}
\Sigma^{\zetaF}&:=\big\{w\in \R^{\mathrm{codim} \, F}:\eklm{{Q_F}w,w}-2\zetaF=0\big\},\\
\Sigma_\bullet^{\zetaF}&:=\big\{w\in \R^{\mathrm{codim} \, F}\setminus\{0\}:\eklm{{Q_F}w,w}-2\zetaF=0\big\}\subset \Sigma^{\zetaF},\\
\Sigma_\times^{\zetaF}&:=\big\{w\in \R^{n_F^+}_\bullet\times \R^{n_F^-}_\bullet:\eklm{{Q_F}w,w}-2\zetaF=0\big\}\subset \Sigma_\bullet^{\zetaF},
\end{align*}
with the notation $\R^{n_F^\pm}_\bullet:=\R^{n_F^\pm}\setminus\{0\}$. 
\label{def:quadrics}
\end{definition}
Note that
\bq
\Sigma_\bullet^{\zetaF}=\Sigma_\times^{\zetaF} \iff \zetaF=0,\quad\qquad \Sigma_\bullet^{\zetaF}=\Sigma^{\zetaF}\iff \zetaF\neq 0.\label{eq:pointedslit}
\eq
The quadric $\Sigma^{\zetaF}$ is the local model for the level set $\J^{-1}(\zeta+\J(F))$ near $F$, with $\zeta=0$ corresponding to the level set of $\J(F)$. Expressing the coefficients in Proposition \ref{prop:09.03.2020} in terms of objects living on $\Sigma^{\zetaF}$ will be the first step towards an intrinsic geometric interpretation of the former. The pointed quadric $\Sigma_\bullet^{\zetaF}$ corresponds to the top stratum of $\J^{-1}(\zeta+\J(F))$, which is reflected by the second relation in \eqref{eq:pointedslit}. Finally, the slit quadric $\Sigma_\times^{\zetaF}$ consists of all points in $\Sigma^{\zetaF}$ which can be described by the inertial polar coordinates introduced in Section \ref{subsec:4.1}.  

There is a natural hypersurface measure on $\Sigma_\bullet^{\zetaF}$ induced by the symplectic measure $dw$ on $\R^{\mathrm{codim} \, F}$, defined by the volume form $d\Sigma^{\zetaF}:=\Theta_{F}|_{\Sigma_\bullet^{\zetaF}}$, where the $(\mathrm{codim} \, F-1)$-form $\Theta_F$ on 
 $\R^{\mathrm{codim} \, F}\setminus\{0\}$ is characterized uniquely near $\Sigma_\bullet^{\zetaF}$ by the relations
\bq
\label{eq:22.10.19}
\d w=\Theta_{F}\wedge\d q_F,\qquad \Theta_{F}|_{w}\in \Lambda^{\mathrm{codim} \, F-1}T^\ast_{w}\Sigma_\bullet^{\zetaF}\subset \Lambda^{\mathrm{codim} \, F-1}T^\ast_{w}\R^{\mathrm{codim} \, F}\quad \forall\; w\in \Sigma_\bullet^{\zetaF},
\eq 
where we wrote $q_F(w):=\frac{1}{2}\eklm{{Q_F} w,w}$. Let $T_F:\R^{\mathrm{codim} \, F}\to\R^{\mathrm{codim} \, F}$ be the isomorphism 
\bq
w=(w_1,\ldots,w_{\mathrm{codim}\,  F})\longmapsto\bigg(\frac{w_1}{|\lambda_1^F|^{\frac{1}{2}}},\frac{w_2}{|\lambda_1^F|^{\frac{1}{2}}},\ldots,\frac{w_{\mathrm{codim}\,  F-1}}{|\lambda_{\mathrm{codim}\,  F/2}^F|^{\frac{1}{2}}},\frac{w_{\mathrm{codim}\,  F}}{|\lambda_{\mathrm{codim}\,  F/2}^F|^{\frac{1}{2}}}\bigg)= T_F(w).\label{eq:TF}
\eq
Pulling back both sides of \eqref{eq:22.10.19} along $T_F$ yields, with $\Lambda_F$ as in \eqref{eq:LambdaF}, the  equation
\bq
\label{eq:pullbackforms11}
\Lambda_F^{-1}\d w=T_F^\ast\d w=T_F^\ast\Theta_{F}\wedge \d (q_F\circ T_F).
\eq 
We claim that, on the subset $\R^{n_F^+}_\bullet\times \R^{n_F^-}_\bullet\subset \R^{\mathrm{codim} \, F}\setminus\{0\}$, the form $T_F^\ast\Theta_{F}$ can be explicitly expressed in 
 terms of the inertial polar coordinates introduced in Section \ref{subsec:4.1} as follows:
\bq
T_F^\ast\Theta_{F}|_{T_F^{-1}(r\theta^+,s\theta^-)}=\Lambda_F^{-1} r^{n_F^+-1} s^{n_F^--1} \d \theta^+ \wedge \d \theta^- \wedge  \frac {r \d s + s \d r}{r^2+s^2},\quad (\theta^+,\theta^-)\in S^{n_F^+-1}\times S^{n_F^--1}.\label{eq:etapolarcoord}
\eq
This follows from \eqref{eq:pullbackforms11}, the equation $(q_F\circ T_F)(r\theta^+,s\theta^-)= \frac{1}{2}(r^2-s^2)$,  and the  equations
\begin{align*}
 d(q_F\circ T_F)\wedge \frac {r \d s + s \d r}{r^2+s^2}&= (r \d r-s \d s) \wedge \frac {r \d s + s \d r}{ r^2+s^2}=\d r \wedge \d s,\\
  dw&=r^{n_F^+-1} s^{n_F^--1}\d r \wedge  \d \theta^+ \wedge\d s\wedge \d \theta^-.
\end{align*}
Since $n^+_F\geq 2$, $n^-_F\geq 2$, and $T_F^\ast d\Sigma^{\zetaF}=T_F^\ast\Theta_{F}|_{T_F^{-1}(\Sigma_\bullet^{\zetaF})}$, we deduce from \eqref{eq:etapolarcoord} that the measure $d\Sigma^{\zetaF}$ is locally finite. Observe that $\Sigma_\times^{\zetaF}$ is of full measure in $\Sigma_\bullet^{\zetaF}$. Thus, we get for every $f\in \CT(\R^{\mathrm{codim} \, F})$
\begin{align}\begin{split}
\intop_{\Sigma_\bullet^{\zetaF}}  f \ d\Sigma^{\zetaF}&=\intop_{T_F^{-1}(\Sigma_\times^{\zetaF})}  f\circ T_F \ T_F^\ast d\Sigma^{\zetaF}\\
&=\Lambda_F^{-1}\intop_{\{r^2-s^2=2\zetaF,\; r,s>0\} } \frac{ r^{n^+_F}  s^{n_F^--1} S_{f\circ T_F} (r,s)\d s +   r^{n^+_F-1}  s^{n_F^-}S_{f\circ T_F} (r,s) \d r}{r^2+s^2}.\label{eq:inttransf}\end{split}
\end{align}
Note that on the integration domain in the second line we have the relation $rdr=sds$. 
 
Next, consider for $k\in \N_0$ the function $W_{F,k}\in \Cinft(\R^{n_F^+}_\bullet\times \R^{n_F^-}_\bullet)$ defined by 
\bq
W_{F,k}(w):=4\Lambda_F\norm{T_F^{-1}w}^{2k}\norm{T_F^{-1}w^+}^{2-n^+_F}\norm{T_F^{-1}w^-}^{2-n^-_F},\label{eq:WFk}
\eq
where we use the notation $w=(w^+,w^-)$ for an element in $\R^{n_F^+}\times \R^{n_F^-}=\R^{\mathrm{codim} \, F}$. Furthermore, we define a differential operator $D^+_F:\Cinft(\R^{n_F^+}_\bullet\times \R^{n_F^-})\to \Cinft(\R^{n_F^+}_\bullet\times \R^{n_F^-})$ and a differential operator $D^-_F:\Cinft(\R^{n_F^+}\times \R^{n_F^-}_\bullet)\to \Cinft(\R^{n_F^+}\times \R^{n_F^-}_\bullet)$ by
\bq
D^\pm_F(f)(w):=\frac{1}{2}\Big\langle\nabla f(w),\frac{w^\pm}{\norm{T_F^{-1}w^\pm}^2}\Big\rangle,\label{eq:Dpm}
\eq
where $\nabla f(w)\in \R^{\mathrm{codim} \, F}$ is the Euclidean gradient of $f$ at $w$ and $\eklm{\cdot,\cdot}$ is the standard inner product in $\R^{\mathrm{codim} \, F}$. The significance of these operators lies in the following observation, which we shall in fact only use for $\zetaF=0$, when the pointed and slit quadrics agree. 

\begin{prop}\label{prop:Dpm}Let $f\in \CT(\R^{\mathrm{codim} \, F})$ and let $f_{\times}$ be the restriction  of $f$ to $\R^{n_F^+}_\bullet\times \R^{n_F^-}_\bullet$. For each $k,l\in \N_0$, the function $W_{F,k} (D^+_F-D^-_F)^l f_{\times}$ is integrable over $\Sigma_\times^{\zetaF}$ with respect to $d\Sigma^{\zetaF}$ and one has
\[
\int_{\Sigma_\times^{\zetaF}}  W_{F,k} (D^-_F-D^+_F)^l f_{\times}\d\Sigma^{\zetaF}=\int_{|2\zetaF|}^{\infty}t^{k} (\partial_--\partial_+)^l\mathscr{S}_{f\circ T_F}\bigg(\frac{t+2\zetaF}{2},\frac{t-2\zetaF}{2}\bigg)\d t.
\]
Consequently, the integrals on the right hand side have a geometric meaning in the sense that they correspond to integrals over $\Sigma_\times^{\zetaF}$ with respect to the natural hypersurface measure  $d\Sigma^{\zetaF}$. 
\end{prop}
\begin{proof}
Let us first assume $\zetaF\geq 0$ and introduce the short-hand notation $f_{k,l}:=W_{F,k} (D^-_F-D^+_F)^l f_{\times}$  with the function $W_{F,k}$ from \eqref{eq:WFk} and the operators $D^\pm_F$ from \eqref{eq:Dpm}. 
The spherical mean $S_{f_{k,l}\circ T_F}(r,s)$ is well-defined for $r,s>0$ and can be expressed in terms of $S_{f\circ T_F}(r,s)$ according to 
\bq S_{f_{k,l}\circ T_F}(r,s)=4\Lambda_F\frac{(r^2+s^2)^{k}}{r^{n^+_F-2}\,s^{n^-_F-2}}\Big(\frac{1}{2s}\frac{\partial}{\partial s}-\frac{1}{2r}\frac{\partial}{\partial r}\Big)^l S_{f\circ T_F}(r,s),\label{eq:Sftildef}
\eq
as can be seen by writing $w$ and $w^\pm$ in terms of polar coordinates in the definitions of $W_{F,k}$ and $D^\pm_F$. 
Using on $\Sigma^{\zetaF}_\times$ the relations $r=\sqrt{s^2+2\zetaF}$ and $rdr=sds$, we write with \eqref{eq:inttransf}
\begin{align*}
&\int_{\Sigma^{\zetaF}_\times}  f_{k,l}\ d\Sigma^{\zetaF}\\
&=\frac{1}{\Lambda_F}\int_{0}^{\infty} \frac{ (s^2+2\zetaF)^{n^+_F/2}  s^{n_F^--1} S_{f_{k,l}\circ T_F} \big(\sqrt{s^2+2\zetaF},s\big)\d s +   (s^2+2\zetaF)^{n^+_F/2-1}  s^{n_F^-+1}S_{f_{k,l}\circ T_F} \big(\sqrt{s^2+2\zetaF},s\big) \d s}{2s^2+2\zetaF}\end{align*}
and then compute with \eqref{eq:Sftildef}
\begin{align*}
\int_{\Sigma^{\zetaF}_\times}  f_{k,l}\ d\Sigma^{\zetaF}&=\frac{1}{\Lambda_F}\int_{0}^{\infty} (s^2+2\zetaF)^{n^+_F/2-1}  s^{n_F^--1}S_{f_{k,l}\circ T_F} \big(\sqrt{s^2+2\zetaF},s\big)\d s\\
&=4\int_{0}^{\infty}s(2s^2+2\zetaF)^{k} \Big(\frac{1}{2s}\partial_s-\frac{1}{2\sqrt{s^2+2\zetaF}}\partial_r\Big)^l S_{f\circ T_F} \big(\sqrt{s^2+2\zetaF},s\big)\d s\\
&=2\int_{0}^{\infty}(2t+2\zetaF)^{k} \Big(\frac{1}{2\sqrt{t}}\partial_s-\frac{1}{2\sqrt{t+2\zetaF}}\partial_r\Big)^lS_{f\circ T_F} \big(\sqrt{t+2\zetaF},\sqrt{t}\big)\d t\\
&=\int_{2\zetaF}^{\infty}t^{k} \Bigg(\frac{1}{2\sqrt{\frac{t-2\zetaF}{2}}}\partial_s-\frac{1}{2\sqrt{\frac{t+2\zetaF}{2}}}\partial_r\Bigg)^lS_{f\circ T_F} \Bigg(\sqrt{\frac{t+2\zetaF}{2}},\sqrt{\frac{t-2\zetaF}{2}}\Bigg)\d t\\
&=\int_{2\zetaF}^{\infty}t^{k} (\partial_--\partial_+)^l\mathscr{S}_{f\circ T_F}\bigg(\frac{t+2\zetaF}{2},\frac{t-2\zetaF}{2}\bigg)\d t.
\end{align*}
Here we substituted $t:=s^2$ and applied \eqref{eq:derivrel} in the final step. The integrability claim follows since the right hand side is a finite integral. 
The case $\zetaF\leq 0$ is treated analogously by writing $s=\sqrt{r^2-2\zetaF}$.
\end{proof}
A further important observation is the following
\begin{prop}\label{prop:diffopint0}For all $k\in \N_0$ and $f\in \CT(\R^{\mathrm{codim} \, F})$, one has 
\bqn
\partial^k_\pm\mathscr{S}_{f\circ T_F}(0,0)=(2n^\pm_F)^{-k}(\mathrm{vol}\,S^{n^+_F-1})(\mathrm{vol}\,S^{n^-_F-1})(\Delta^{Q_F}_\pm)^k f (0),
\eqn
where $\mathrm{vol}(S^{n^\pm_F-1})$ is the volume of  $S^{n^\pm_F-1}$ with respect to the standard round measure and the second order differential operator $\Delta^{Q_F}_\pm:\Cinft(\R^{\mathrm{codim} \, F})\to \Cinft(\R^{\mathrm{codim} \, F})$ is defined by
\[
\Delta^{Q_F}_\pm(f)(w):=\mathrm{tr}\,(Q_F^\pm)^{-1}\mathrm{Hess}_\pm(f)(w),\qquad w\in \R^{\mathrm{codim} \, F}.
\]
Here $\mathrm{tr}$ denotes the trace, $(Q_F^\pm)^{-1}$ is the restriction of $Q_F^{-1}$  to the subspace $\R^{n_F^\pm}\subset\R^{\mathrm{codim} \, F}$ on which $\pm Q_F$ is positive, and $\mathrm{Hess}_\pm(f)(w)$ denotes the quadrant of the Euclidean Hessian matrix of $f$ at $w$  formed by all second derivatives with respect to the variables in $\R^{n_F^\pm}$.
\end{prop}
\begin{proof} For $k=0$, the claim is true by  \eqref{eq:doublesphermean}. Assuming that it  holds for some $k\in \N_0$ and also for $k=0$, we have
\[
\partial^{k+1}_\pm\mathscr{S}_{f\circ T_F}(0,0)=\partial_\pm\partial^k_\pm\mathscr{S}_{f\circ T_F}(0,0)=\partial_\pm\mathscr{S}_{(2\,n^\pm_F)^{-k}(\Delta^{Q_F}_\pm)^k (f\circ T_F)}(0,0).
\]
This reduces the proof to the case $k=1$. To treat this case, we first recall that for any bilinear form $B$ on any Euclidean space $\R^n$ one has 
\bq
\int_{S^{n-1}}\eklm{B\theta,\theta}\d\theta=\frac{1}{n}\mathrm{vol}\,S^{n-1}\,\mathrm{tr}\,B,\label{eq:bilineartrace}
\eq
as one verifies by diagonalizing $B$. For arbitrary $f\in \CT(\R^{\mathrm{codim} \, F})$, we now  compute using \eqref{eq:derivrel} 
\begin{align*}
&\partial_+\mathscr{S}_{f\circ T_F}(0,0)=\lim_{\eps \to 0^+}\partial_+\mathscr{S}_{f\circ T_F}(\eps^2,0)=\lim_{\eps \to 0^+}\frac{1}{2\eps}\frac{\partial S_{f\circ T_F}}{\partial r}(\eps,0)\\
&=\lim_{\eps \to 0^+}\frac{1}{2\eps}\int_{S^{n_F^+-1}} \int_{S^{n_F^--1}}   \eklm{\nabla_+ f(\eps T_F\theta^+, 0),T_F\theta^+} \d\theta^+ \d \theta^- \\
&=\lim_{\eps \to 0^+}\frac{\mathrm{vol}\,S^{n^-_F-1}}{2\eps}\bigg(\int_{S^{n_F^+-1}}\eklm{\nabla_+ f(0,0), T_F\theta^+}\d \theta^++\eps\int_{S^{n_F^+-1}}\eklm{\mathrm{Hess}_+ f(0,0) T_F\theta^+, T_F\theta^+}\d \theta^++\O(\eps^2)\bigg)\\
&=\frac{\mathrm{vol}\,S^{n^-_F-1}}{2}\int_{S^{n_F^+-1}}\eklm{T_F\mathrm{Hess}_+ f(0,0) T_F\theta^+, \theta^+}\d \theta^+ 
\end{align*}
because $\int_{S^{n_F^+-1}}\eklm{v,\theta^+}\d \theta^+=0$ for each $v\in \R^{n_F^+}$ and $T_F$ is self-adjoint. Applying \eqref{eq:bilineartrace} then yields 
\bqn
\partial_+\mathscr{S}_{f\circ T_F}(0,0)=(2n^+_F)^{-1}(\mathrm{vol}\,S^{n^+_F-1})(\mathrm{vol}\,S^{n^-_F-1})\,\mathrm{tr}(T_F^\pm\mathrm{Hess}_+ f(0,0) T_F^\pm),
\eqn
where we put $T_F^\pm:=T_F|_{\R^{n_F^\pm}}:\R^{n_F^\pm}\to \R^{n_F^\pm}$. 
To get the claimed relation, it now suffices to use the cyclic property of the trace and to observe that $(T_F^\pm)^2=(Q_F^\pm)^{-1}$.  The calculation for $\partial_-\mathscr{S}_f(0,0)$ is completely analogous.
\end{proof}

To close this section, we briefly turn to the case of a definite quadratic form $Q_F$ and write
\bq
\label{eq:doublesphermeandef}
 S_f(r):=\int_{S^{\mathrm{codim}\,F -1}}  f (r  \theta) \d\theta=:\mathscr{S}_f(r^2),\qquad r\in \R.
\eq
\begin{prop}\label{prop:extD2}One has for all $k\in \N_0$ and $f\in \CT(\R^{\mathrm{codim} \, F})$
\[
\mathscr{S}^{(k)}_{f\circ T_F}(0)=\frac{\mathrm{vol}(S^{\mathrm{codim} \, F-1})}{(2\,\mathrm{codim}\, F)^k}  (\Delta^{Q_F})^k f (0),
\]
where $\mathrm{vol}(S^{\mathrm{codim} \, F-1})$ is the volume of  $S^{\mathrm{codim} \, F-1}$ with respect to the standard round measure and the second order differential operator $\Delta^{Q_F}:\Cinft(\R^{\mathrm{codim} \, F})\to \Cinft(\R^{\mathrm{codim} \, F})$ is defined by
\[
\Delta^{Q_F}(f)(w):=\mathrm{tr}\,Q_F^{-1}\mathrm{Hess}(f)(w),\qquad w\in \R^{\mathrm{codim} \, F}.
\]
Here $\mathrm{tr}$ denotes the trace and $\mathrm{Hess}(f)(w)$ denotes the Euclidean Hessian of $f$ at the point $w$.
\end{prop}
\begin{proof}In view of the relations \bqn
\mathscr{S}^{(k)}_f(t)=\delta_r^k S_f(\sqrt{t})\quad \forall\; t\in (0,\infty),\; k\in \N_0,\qquad \delta_r S_f(r):=\frac{1}{2r} S_f'(r),\quad r>0,
\eqn
the proof is completely analogous to the proof of Proposition \ref{prop:diffopint0}. 
\end{proof}

\subsection{Global geometric interpretation}\label{sec:geomint2}
Let us now carry out the second step of our geometric interpretation by specializing from a general spherical mean $S_f$ as in  \eqref{eq:doublesphermean} and \eqref{eq:doublesphermeandef} to the particular spherical means $S_F$ from \eqref{eq:SFrs} and \eqref{eq:RSa} which involve our amplitude $a$, the local normal form symplectomorphism $\Phi_F$ from Proposition \ref{prop:localnormform}, and the cutoff function $\chi_F$ from our partition of unity. Our goal is to translate the expressions from Propositions \ref{prop:Dpm},  \ref{prop:diffopint0}, and  \ref{prop:extD2}, which live in our local model of $M$ near $F$, into expressions that live on subsets of the symplectic reduction $\M^\zeta$. The technical key ingredient to achieving our goal is the following 

\begin{lem}\label{lem:geometric_interpretation} Fix  $(U_F,\Phi_F)$ as in Proposition \ref{prop:localnormform}, let $f\in \mathrm{C}_c(U_F)$, and define the average
$$
\eklm{f}_T(T\cdot  p):=\int_{T}f(g\cdot p)\d g,\qquad   p \in M, 
$$
where $d g$ is the Haar measure on $T$ fixed by our identification $\t\cong\R$. 
 Let $\zeta\in \R$ and put $\zeta_F=\J(F)-\zeta$.
\begin{enumerate}[leftmargin=*]
\item  Let $d\Sigma^{\zeta_F}$ be the hypersurface measure introduced in \eqref{eq:22.10.19} on the pointed quadric $\Sigma^{\zeta_F}_\bullet$. Then, with the notation as in \eqref{eq:meas2}, one has
  \bqn
    \int_{\M^{\zeta}_\mathrm{top}}^{}  \eklm{f}_T
   \d\M^{\zeta}_\mathrm{top}= \int_{P_F} 
     \int_{\Sigma^{\zeta_F}_\bullet} f (\Phi^{-1}_{F}(\tilde \pi_F(\wp,w))) 
    \d\Sigma^{\zeta_F}(w) \d \wp,
  \eqn
where  $\d\M^{\zeta}_\mathrm{top}={(\omega^\zeta_{\mathrm{top}})^{n-1}}/{(n-1)!}$ denotes the  symplectic volume form on the top stratum of $\M^{\zeta}$.

\item Similarly, we have with $dF ={\omega^{\dim F/2}}/(\dim F/2)!$
 \[
   \int_F   f\;
   dF= \int_{P_F}   f (\Phi^{-1}_{F}(\tilde \pi_F(\wp,0))) \d \wp.
  \]
\end{enumerate}
\end{lem}
\begin{proof} We begin with Assertion (1). First, note that from \eqref{eq:30.06.2018} and \eqref{eq:22.10.19a} it is clear that
$$
\Phi_F(\J^{-1}(\{\zeta\})_{\mathrm{top}}\cap U_F)= \mklm{[\wp,w] \in \Phi_F(U_F) \mid w \in \Sigma_\bullet^{\zeta_F} \setminus\{0\} }.
$$ 
By slight abuse of notation, let us write $q_F([\wp,w]):=\frac{1}{2}\eklm{Q_Fw,w}=:q_F(w)$, so that we deduce from \eqref{eq:30.06.2018} that 
\bq
(\Phi_F^{-1})^\ast d \J=dq_F.\label{eq:-69437-49-5078347}
\eq Furthermore, recall  from \eqref{eq:integrallift} that we have the equality of smooth densities
\bqn
(\Phi_F^{-1})^\ast(\d M|_{U_F}) =  \d[\wp,w], \qquad  |\tilde \pi_F^*(d[\wp,w]) \wedge \Pi_{F}^\ast\eta_F|=|d \wp \wedge \d w|
\eqn  
and that by \eqref{eq:22.10.19} we have $d\wp \wedge dw=d\wp \wedge \Theta_{F}\wedge\d q_F $.  On the other hand, the hypersurface Liouville measure $d\J^{-1}(\{\zeta\})_{\textrm{top}}$ on $\J^{-1}(\{\zeta\})_{\textrm{top}}$ is characterized by the condition 
\bq
\d M|_p= \frac{1}{n!}\omega^n|_p= d \J|_p  \wedge d\J^{-1}(\{\zeta\})_{\textrm{top}}|_p 
\,  \in \Lambda^{2n}(T^\ast_p M)\label{eq:390239583209}
\eq
for any  $p\in \J^{-1}(\{\zeta\})_{\textrm{top}}$. Pulling back \eqref{eq:390239583209} along $\Phi_F^{-1}\circ \tilde \pi_F$ and taking \eqref{eq:-69437-49-5078347} into account yields 
\[
|(\Phi_F^{-1}\circ \tilde \pi_F)^\ast(d\J^{-1}(\{\zeta\})_{\textrm{top}})\wedge \Pi_{F}^\ast\eta_F| = |\d \wp  \wedge \Theta_F|.
\]
Since $\Theta_F|_{\Sigma_\bullet^{\zeta_F}}=d\Sigma_\bullet^{\zeta_F}$  and $\int_{\pi_F^{-1}(\{p\})}\eta_F=1$ for each $p\in F$, this proves that one has for each $f\in C_c(U_F)$
\bq
    \int_{\J^{-1}(\{\zeta\})_{\textrm{top}}}^{}   f\;
   \d\J^{-1}(\{\zeta\})_{\textrm{top}}= \int_{P_F} 
     \int_{\Sigma_\bullet^{\zeta_F}} f (\Phi^{-1}_{F}(\tilde \pi_F(\wp,w)))
    \d\Sigma_\bullet^{\zeta_F}(w) \d \wp. \label{eq:3930-6930-326326326}
  \eq
 It remains to show that   
  \bq
  \int_{\M^{\zeta}_\mathrm{top}}^{}  \eklm{f}_T
   \d\M^{\zeta}_\mathrm{top}=\int_{\J^{-1}(\{\zeta\})_{\textrm{top}}}^{}   f\;
   \d\J^{-1}(\{\zeta\})_{\textrm{top}}.\label{eq:9326802396803297}
\eq
To this end, recall that $\omega^\zeta_{\textrm{top}}$ is characterized by $\pi^\ast\omega^\zeta_{\textrm{top}} = i_\zeta^\ast\omega$, where $i_\zeta : \J^{-1}(\{\zeta\})_{\textrm{top}} \to
  M$ is the inclusion and $\pi: \J^{-1}(\{\zeta\})_{\textrm{top}} \to\M^{\zeta}_\mathrm{top}$ the orbit projection. Also notice  that our identification of $t$ with $\R$ corresponds to a choice of an element $x_0\in \t$ that is identified with $1$, and leads to an identification of $\J$ with $J(x_0)$. On the top stratum $ M_{(\h_\mathrm{top})}\subset M$, the fundamental vector field $\tilde x_0$ is nowhere-vanishing, so it has a dual one-form $\xi_{0}$. One then computes on $M_{(\h_\mathrm{top})}$
  \bqn
\iota_{\tilde x_0}(\xi_{0}\wedge d \J)=d \J=d J(x_0)=\iota_{\tilde x_0}\omega,
\eqn
where the first equality uses the $T$-invariance of $\J$, the middle equality is the remark above, and the last equality is the defining property  of the momentum map $\J$. Consequently, 
\bq
(\xi_{0}\wedge d \J)|_{M_{(\h_\mathrm{top})}} + \beta=\omega|_{M_{(\h_\mathrm{top})}}
\label{eq:21851290580}
  \eq
  for some $\beta\in \Omega^2(M_{(\h_\mathrm{top})})$ that fulfills $\iota_{\tilde x_0}\beta=0$.  Thus, on $M_{(\h_\mathrm{top})}$ we have
   \begin{align*}
 \frac{1}{n!}\omega^n&=\frac{1}{n!}\omega^{n-1}\wedge \xi_{0}\wedge d \J + \frac{1}{n!}\omega^{n-1}\wedge \beta\\
 &=\frac{1}{n!}\omega^{n-1}\wedge \xi_{0}\wedge d \J + \frac{1}{n!}\omega^{n-2}\wedge \xi_{0}\wedge d \J \wedge \beta + \frac{1}{n!}\omega^{n-2}\wedge \beta^2\\
 &=\frac{1}{n!}\bigg(\sum_{j=1}^{n}\omega^{n-j}\wedge\beta^{j-1}\bigg)\wedge \xi_{0}\wedge d \J + \frac{1}{n!}\beta^n
 \end{align*}
 but $\beta^n=0$ because $\beta$ is degenerate.   Now, if we insert $\beta=\omega-\xi_{0}\wedge d \J$, then all non-zero powers of $\xi_{0}\wedge d \J$ get killed by the  wedge product with $\xi_{0}\wedge d \J$, and we arrive at
    \bqn
 \frac{1}{n!}\omega^n=\frac{1}{n!}\bigg(\sum_{j=1}^{n}\omega^{n-1}\bigg)\wedge \xi_{0}\wedge d \J=\frac{1}{(n-1)!}\omega^{n-1}\wedge \xi_{0}\wedge d \J.
 \eqn 
Inserting this in \eqref{eq:390239583209} gives us
  \[
  d\J^{-1}(\{\zeta\})_{\textrm{top}} = \frac{1}{(n-1)!}i_\zeta^\ast\omega^{n-1}\wedge i_\zeta^\ast(\xi_{0}).
  \]
On the other hand, we compute
  \[
\pi^\ast \d\M^{\zeta}_\mathrm{top}=\pi^\ast\big({(\omega^\zeta_{\mathrm{top}})^{n-1}}/{(n-1)!}\big)=\frac{1}{(n-1)!}\pi^\ast(\omega^\zeta_{\mathrm{top}})^{n-1}=\frac{1}{(n-1)!}i_\zeta^\ast\omega^{n-1},
\]
so that we find
\begin{align*}
d\J^{-1}(\{\zeta\})_{\textrm{top}}&=\pi^\ast(\d\M^{\zeta}_\mathrm{top})\wedge i_\zeta^\ast(\xi_{0}),\\
 \int_{\J^{-1}(\{\zeta\})_{\textrm{top}}}^{}   f\;
   \d\J^{-1}(\{\zeta\})_{\textrm{top}}&= \int_{\M^0_{\mathrm{top}}}\left(\int_{T\cdot p}f \xi_{0}\right)\d\M^{\zeta}_\mathrm{top}(T\cdot p).
  \end{align*}
We are left with comparing $\int_{T\cdot p}f \xi_{0}$ with $\eklm{f}_T(T\cdot p)$ for any orbit $T\cdot p$, where $p\in \J^{-1}(\{\zeta\})_{\textrm{top}}$. The map $
\Psi_p:S^1\owns g\mapsto g\cdot p \in T\cdot p$ 
is an $S^1$-equivariant diffeomorphism, so $\Psi_p^\ast \xi_{0}$ is a Haar measure on $S^1$ and thus a constant multiple of $dg$. To determine the constant, we note that the derivative of $\Psi_p$  at the identity fulfills
 $D\Psi_p|_1(x)=\tilde x|_p$,  $x\in \t=T_1S^1$. Consequently, 
 \[
 \Psi_p^\ast \xi_{0}|_1(x_0)=\xi_{0}(D\Psi_p|_1(x_0))=\xi_{0}(\tilde x_0)=1,
 \]
 proving $\Psi_p^\ast \xi_{0}=dg$. This finishes the proof of \eqref{eq:9326802396803297} and Assertion (1).  Assertion (2) follows along the same arguments taking into account \eqref{eq:UFMT}, \eqref{eq:meas1}, and the relation $
 \Phi_F(F)= \mklm{[\wp,w] \in \Phi_F(U_F) \mid w =0 }$.
 \end{proof}

\begin{cor}\label{cor:geom1}  Consider a spherical mean $S_F$ as in \eqref{eq:SFrs} and \eqref{eq:RSa}. Depending on whether the bilinear form $Q_F$ is indefinite or definite, write
\[
S_F(r,s)=:\mathscr S_{\tilde a_F}(r^2,s^2)\qquad \text{or}\qquad S_F(r)=:\mathscr S_{\tilde a_F}(r^2),\qquad r,s\in [0,\infty),
\]
where $\mathscr S_{\tilde a_F}$ is a smooth, compactly supported function on $\R^2$ or $\R$ as in \eqref{eq:doublesphermean} or \eqref{eq:doublesphermeandef}, respectively,  associated with the function $\tilde a_F$ from \eqref{eq:SFrs} which is related to our amplitude $a\in \CT(M)$ by
\[
 \tilde a_F(w) =\int_{P_F} (a\chi_F)(\Phi^{-1}_{F}(\tilde \pi_F(\wp,T_Fw)))  \d \wp,\qquad w\in \R^{\mathrm{codim} \, F}.
\]
\begin{enumerate}[leftmargin=*]
\item Suppose that $Q_F$ is indefinite. Then, for each $k,l\in \N_0$, there are differential operators \bqn
\mathscr{D}_{F,k,l}:\Cinft(U_F\setminus F)\to \Cinft(U_F\setminus F),\qquad 
\mathscr{Q}^\pm_{F,k}: \Cinft(U_F)\to \Cinft(U_F)
\eqn
 of orders $k$ and $2k$, respectively, such that
\begin{align}\begin{split}
\int_{0}^{\infty}t^{l} (\partial_--\partial_+)^k\mathscr{S}_{\tilde a_F}\Big(\frac{t}{2},\frac{t}{2}\Big)\d t&= \int_{\M^{\J(F)}_\mathrm{top}}  \eklm{\mathscr{D}_{F,k,l}(\chi_F a)}_T
   \d\M^{\J(F)}_\mathrm{top},\\
\partial^k_\pm\mathscr{S}_{\tilde a_F}(0,0)&=\int_F \mathscr{Q}^\pm_{F,k} a \d F.\label{eq:int362372}\end{split}
\end{align}
\item Suppose that $Q_F$ is definite. Then, for each $k\in\N_0$, there is a differential operator $\mathscr{Q}_{F,k}: \Cinft(U_F)\to \Cinft(U_F)$ of order $2k$ such that
\bq
 \mathscr{S}^{(k)}_{\tilde a_F}(0)=\int_F \mathscr{Q}_{F,k} a \d F.\label{eq:geom2}
\eq
\end{enumerate}
\end{cor}
\begin{proof}
Applying Lemma \ref{lem:geometric_interpretation} and Propositions \ref{prop:Dpm},  \ref{prop:diffopint0}, \ref{prop:extD2} yields the claim with integrands of the form
\[
\eklm{\mathscr{D}_{F,k,l}(a\chi_F)}_T,\quad \mathscr{Q}^\pm_{F,k} (a\chi_F), \quad \mathscr{Q}_{F,k} (a\chi_F),
\]
where the operators $\mathscr{D}_{F,k,l}$,  $\mathscr{Q}^\pm_{F,k}$, and $\mathscr{Q}_{F,k}$ are defined as follows: first, we extend the operators $D_F^\pm$, $\Delta^{Q_F}$, $\Delta^{Q_F}_\pm$  and the functions $W_{F,k}$ from $V\subset\R^{\mathrm{codim} \, F}$ (here $V$ is $\R^{\mathrm{codim} \, F}$ or $ \R^{n_F^+}_\bullet\times \R^{n_F^-}_\bullet$)   to the trivial bundle $P_F\times V$ by pulling back along the projection $P_F\times V\to V$. Then these extended operators and functions induce differential operators $\tilde D^\pm_F$, $\tilde \Delta^{Q_F}$, $\tilde \Delta_\pm^{Q_F}$ and functions $\tilde W_{F,k}$ on  $P_F\times_{K_F} V$ by identifying smooth functions on $P_F\times_{K_F} V$ 
with $K_F$-invariant smooth functions on $P_F\times V$. By conjugating $\tilde D^-_F-\tilde D^+_F$  with $\Phi_F$, taking the $l$-th power, and multiplying with $\tilde W_{F,k}$ and appropriate constants as prescribed by Proposition \ref{prop:09.03.2020}, we obtain the operator $\mathscr{D}_{F,k,l}$. Similarly, we define the operators  $\mathscr{Q}_{F,k}$ and $\mathscr{Q}^\pm_{F,k}$ using $\tilde\Delta^{Q_F}$ and $\tilde\Delta_\pm^{Q_F}$. Finally, since $\chi_F\equiv 1$ near $F$, we can remove $\chi_F$ from the integrals over $F$.  
\end{proof}

\section{Proof of the main results}
\label{sec:5}

We are now ready to prove our main results, Theorems \ref{thm:main1} and \ref{thm:main2}.

\subsection{Proof of Theorem \ref{thm:main1}}\label{sec:proof}Fix a given $\zeta\in \R\cong \t^\ast$, and recall from \eqref{eq:sumdistr} the equality
\bq
I^\zeta(\varepsilon)  =I_{\chi_\mathrm{top}}^\zeta (\varepsilon)+  \sum_{F\in \mathcal \F}  I_{\chi_F}^\zeta (\varepsilon) \label{eq:decomprep}
\eq
which is equivalent to the statement \eqref{eq:25.5.2019}. We shall derive the desired asymptotic expansion of $I^\zeta(\varepsilon)$ by deriving an asymptotic expansion of each summand in \eqref{eq:decomprep}. Let us begin with the summands associated with fixed point set components $F\in \F$ satisfying $\J(F)=\zeta$. Fix such an $F$, put $j_F:=\mathrm{codim} \, F/ 2 -1$, and recall the notation \eqref{eq:09.03.2020}  as well as the expression \eqref{eq:27.10.2018}.  
Applying Corollary \ref{cor:geom1} to the statements of Propositions \ref{prop:09.03.2020} and \ref{prop:4.3.def}, taking $\mathscr S=\mathscr{S}_{\tilde a_F}$, $\zeta=0$, $2\eps$ as asymptotic parameter, and $L=\mathrm{codim} \, F/ 2 -2$ (in Proposition \ref{prop:09.03.2020}) or $L=\mathrm{codim} \, F/ 2$ (in Proposition \ref{prop:4.3.def}) there,  we obtain the following results:
\begin{enumerate}[leftmargin=*]\item Suppose that $F$ is indefinite. Then, there are differential operators $D^\mathrm{top}_{j,F}:\Cinft(U_F\setminus F)\to \Cinft(U_F\setminus F)$, $j\in \N_0$, and 
$D^\pm_{j,F}: \Cinft(U_F)\to \Cinft(U_F)$, $j\geq j_F$, 
 of orders $j$ and $2(j+1)-\mathrm{codim} \, F$ respectively, such that one has
\bqn
I_{\chi_F }^{\J(F)} (\eps)\;
\sim\;\eps\sum_{j=0}^{\infty}\eps^j A^{\J(F)}_{j,F} ,   \qquad \text{in}\quad(\D(M)\otimes\S(\t))'
\eqn
with $A^{\J(F)}_{j,F}=A'_{j,F} + A_{j,F}$ given by 
\begin{align*}
A'_{j,F}(a\otimes\sigma)&=\sigma^{(j)}(0)\int_{\M^{\J(F)}_\mathrm{top}} \eklm{D^\mathrm{top}_{j,F}(\chi_F a)}_T\d\M^{\J(F)}_\mathrm{top},\\
A_{j,F}(a\otimes\sigma)&=\begin{cases} 0, & j<j_F, \\ \sigma^{[k]}_{+}(0)\int_F D^+_{j,F} a \d F+\sigma^{[k]}_{-}(0)\int_F D^-_{j,F} a \d F, & j\geq j_F.
\end{cases}
\end{align*}
\item Suppose that $F$ is definite, $s_F=\pm$. Then, there are differential operators 
$D_{j,F}:\Cinft(U_F)\to \Cinft(U_F)$, $j\geq \frac{1}{2}\mathrm{codim} \, F-1$, of orders $2j+2-\mathrm{codim} \, F$ such that one has
\bqn
I_{\chi_F}^{\J(F)} (\eps)\;
\sim\;\eps\sum_{j=0}^{\infty}\eps^j A_{j,F}^{\J(F)}  \qquad \text{in}\quad(\D(M)\otimes\S(\t))',
\eqn
where $A_{j,F}^{\J(F)}$ vanishes unless $j \geq j_F$, in which case one has
$
A_{j,F}^{\J(F)}(a\otimes \sigma)=A_{j,F}(a\otimes \sigma)=\sigma^{[j]}_{\mp}(0)\int_F D_{j,F} a \d F.
$
\end{enumerate}
To prove that the operators $D^\pm_{j,F}$, $D_{j,F}$ equal constants for $j=\frac{1}{2}\mathrm{codim} \, F-1$ and to determine the constants $C_{\mathrm{indef}}$, $C_{\mathrm{def}}\in \C$  such that $D^\pm_{\frac{1}{2}\mathrm{codim} \, F-1,F}=C^\pm_{\mathrm{indef}}$, $D_{\frac{1}{2}\mathrm{codim} \, F-1,F}=C_{\mathrm{def}}$, we inspect the differential operators and the constants occurring in Propositions \ref{prop:09.03.2020}, \ref{prop:4.3.def},  \ref{prop:diffopint0}, and \ref{prop:extD2}, and recall that \eqref{eq:27.10.2018} and \eqref{eq:27.10.2018a} involve an overall factor of $\Lambda_F^{-1}$, which gives us 
\begin{align*}
C^\pm_{\mathrm{indef}}&=\Lambda_F^{-1}(\mathrm{vol}\,S^{n^+_F-1})(\mathrm{vol}\,S^{n^-_F-1})2^{\mathrm{codim}\, F/2}c^\pm_{\mathrm{codim}\, F/2-1,0,0,0}\\
&=\frac{2^{\mathrm{codim}\, F/2+2}\pi^{\mathrm{codim}\, F/2}}{\Lambda_F(n^+_F/2-1)!(n^-_F/2-1)!}c^\pm_{\mathrm{codim}\, F/2-1,0,0,0},\\
C_{\mathrm{def}}&=\Lambda_F^{-1}\mathrm{vol}(S^{\mathrm{codim} \, F-1})2^{\mathrm{codim}\, F/2} c_{\mathrm{codim}\, F/2-1,0}=\frac{2^{\mathrm{codim}\, F/2+1}\pi^{\mathrm{codim}\, F/2}}{\Lambda_F(\mathrm{codim}\, F/2-1)!} c_{\mathrm{codim}\, F/2-1,0}.
\end{align*}
Here the additional factors $2^{\mathrm{codim}\, F/2}$ occur because we apply Propositions \ref{prop:09.03.2020} and \ref{prop:4.3.def} with $\eps$ replaced by $2\eps$. Similarly, one finds $D^\mathrm{top}_{0,F}=2\pi$ by taking into account that the function $W_{F,k}$ occurring in Proposition \ref{prop:Dpm} is constant when $k=\mathrm{codim}\, F/2-2$, its value being given by $W_{F,\mathrm{codim}\, F/2-2}=2^{\mathrm{codim}\, F/2}\Lambda_F$ since $\norm{w}^2=2\norm{w^\pm}^2$ for $w\in T_F^{-1}(\Sigma^0)$. Plugging in the values of the constants $c^\pm_{\mathrm{codim}\, F/2-1,0,0,0}$, and $c_{\mathrm{codim}\, F/2-1,0}$, we arrive at
\bqn
C_{\mathrm{def}}=(2\pi)^2\frac{(2\pi i)^{\mathrm{codim}\, F/2-1}}{\Lambda_F(\mathrm{codim}\, F/2-1)!},\qquad 
C^\pm_{\mathrm{indef}}=(2\pi)^2\frac{(\pi i)^{\mathrm{codim}\, F/2-1}}{\Lambda_F(\mathrm{codim}\, F/2-1)!} N^\pm_F,
\eqn
and we also find that the constants $N^\pm_F$ from \eqref{eq:constantsDF} are given by
\begin{multline}
N^\pm_F=\frac{(-1)^{\mathrm{codim}\, F/2}(\mathrm{codim}\, F/2-1)!}{(n^+_F/2-1)!(n^-_F/2-1)!}\sum_{l=0}^{\mathrm{codim}\, F/2-2}\frac{(\pm 1)^{\mathrm{codim}\, F/2-l-1}}{\mathrm{codim}\, F/2-l-1}\label{eq:Nvalues0}\\
\sum_{\substack{l^++l^-=l\\ 0\leq l^\pm\leq n^\pm_F/2-1}}(-1)^{l^+}{n^+_F/2-1 \choose {l^+}}{n^-_F/2-1 \choose {l^-}}.
\end{multline}
To simplify the formula for $N^\pm_F$, we perform a computation kindly suggested by  Iosif Pinelis \cite{pinelisMO}. One computes for arbitrary $p,q\in \N_0$ with the convention that $\binom ab=0$ whenever $a<b$
\begin{align}
\nonumber	\sum_{l\ge0}\frac{1}{p+q-l+1}\sum_{\substack{j+k=l\\ j\ge0,k\ge0}}(-1)^{j}\binom pj\binom qk  
\nonumber	&=\sum_{l\ge0}\int_0^1 \,x^{p+q-l}\sum_{\substack{j+k=l\\ j\ge0,k\ge0}}(-1)^{j}\binom pj\binom qk \d x \\ 
\nonumber	&=\int_0^1 \,x^{p+q}\sum_{l\ge0}\sum_{\substack{j+k=l\\ j\ge0,k\ge0}}(-1)^{j}\binom pj\binom qk x^{-j} x^{-k} \d x\\
	&=\int_0^1 \,x^{p+q}\sum_{j\ge0}\binom pj(-x^{-1})^j\;\sum_{k\ge0}\binom qk x^{-k} \d x \label{eq:simplif}\\
\nonumber		&=\int_0^1 \,x^{p+q}(1-x^{-1})^p(1+x^{-1})^q \d x=\int_0^1 \,(x-1)^p(1+x)^q \d x,  
\end{align} 
	and the latter expression can be further rewritten into
\begin{align*}	
	\int_0^1 \,(x-1)^p\,\sum_{k=0}^q\binom qk x^k \d x &=(-1)^p\sum_{k=0}^q\binom qk\int_0^1 \,(1-x)^p x^k\d x \\ 
	&=(-1)^p\sum_{k=0}^q\frac{q!}{k!(q-k)!}\frac{k!p!}{(k+p+1)!}=(-1)^p \frac{q!p!}{(q+p+1)!} \sum_{j=0}^q\binom{q+p+1}j. 
	\end{align*} 
Applying this to \eqref{eq:Nvalues0} gives us the result\footnote{The appearance of $n^-_F$ rather than $n^\mp_F$ in the exponent determining the overall sign of $N^\pm_F$ is not a misprint. This subtle and seemingly peculiar asymmetry is a consequence of the fact that interchanging $n_F^+$ and $n^-_F$ corresponds to replacing the momentum map $\J$ by $-\J$, which changes the phase function in the generalized Witten integral \eqref{eq:2}.}
\bq
N^\pm_F=\pm (-1)^{n^-_F/2-1}  \sum_{j=0}^{n^\mp_F/2-1}\binom{\mathrm{codim}\, F/2-1}j .\label{eq:Nvalues}
\eq
In particular, we see that $N^\pm_F$ is a non-zero integer.
Let us now turn to the remaining summands in \eqref{eq:decomprep}:
\begin{enumerate}[leftmargin=*]
\item[(3)]  Assume that $F\in \F$ is such that $\J(F)\neq \zeta$. Then $\zeta$ is a regular value of $\J|_{U_F}$, and by Proposition \ref{prop:asymptreg} there is  an asymptotic expansion
\begin{align*}
I_{\chi_F}^{\zeta} (\eps)\;
&\sim\;\eps\sum_{j=0}^{\infty}\eps^j A^\zeta_{j,F} \qquad \text{in}\quad(\D(M)\otimes\S(\t))'
\end{align*}
with 
$$
A^\zeta_{j,F}(a\otimes \sigma)=\sigma^{(j)}(0)\int_{\mathscr M^\zeta_\mathrm{top}}\big<\mathscr D^\zeta_{j} (\chi_F a)\big>_T\d \mathscr M^\zeta_\mathrm{top},\qquad \mathscr D^\zeta_{0}=2\pi.
$$
Of course, instead of applying Proposition \ref{prop:asymptreg} we could have also treated this case as in (1) or (2) by applying  Corollary \ref{cor:geom1} and Propositions \ref{prop:09.03.2020} or \ref{prop:4.3.def}, respectively. By the uniqueness of the coefficients in the asymptotic expansion, the results of the two approaches agree. However, the form of the coefficients obtained in Proposition \ref{prop:asymptreg} is much more simple.
\item[(4)] The function $\J|_{M_{\h_\mathrm{top}}}$ has only regular values. In particular, $\zeta$ is a regular value  and  Proposition \ref{prop:asymptreg} gives us  an asymptotic expansion
\begin{align*}
I_{\chi_\mathrm{top}}^{\zeta} (\eps)\;
&\sim\;\eps\sum_{j=0}^{\infty}\eps^j A^\zeta_{j,\mathrm{top}} \qquad \text{in}\quad(\D(M)\otimes\S(\t))'
\end{align*}
with 
$$
A^\zeta_{j,\mathrm{top}}(a\otimes \sigma)=\sigma^{(j)}(0)\int_{\mathscr M^\zeta_\mathrm{top}}\big<\mathscr D^\zeta_{j} (\chi_\mathrm{top} a)\big>_T\d \mathscr M^\zeta_\mathrm{top},\qquad \mathscr D^\zeta_{0}=2\pi.
$$
\end{enumerate}
Taking (1)--(4)  together we deduce  with \eqref{eq:decomprep} that 
\bqn
I^\zeta(\varepsilon)  \sim \eps\sum_{j=0}^\infty \eps^j A^\zeta_j\qquad \text{in}\quad (\D(M)\otimes \S(\t))'
\eqn
with
\bqn
A^\zeta_j(a\otimes\sigma)=\sigma^{(j)}(0)\int_{\mathscr M^\zeta_{\mathrm{top}}}\big<D^\zeta_{j} a\big>_T\d \mathscr M^\zeta_{\mathrm{top}}+\sum_{\substack{F\in \F:\J(F)=\zeta,\\F\cap \supp a\neq \emptyset}} A_{j,F}(a\otimes\sigma),
\eqn
where the differential operator $D^\zeta_{j}$ is defined on the neighborhood $\mathscr U_j^\zeta\subset M_{(\h_\mathrm{top})}$ of $\J^{-1}(\{\zeta\})\cap M_{(\h_\mathrm{top})}$ on which the operator $\mathscr D^\zeta_{j} $ from Proposition \ref{prop:asymptreg} is defined and acts on a function $f\in \Cinft(\mathscr U_j^\zeta)$ by
\bq
D^\zeta_{j}(f):=\mathscr D^\zeta_{j} (\chi_\mathrm{top}|_{\mathscr U_j^\zeta} f) + \sum_{F\in \F:\J(F)\neq \zeta} \mathscr D^\zeta_{j}(\chi_F|_{\mathscr U_j^\zeta} f)+ \sum_{\substack{F\in \F:\J(F)=\zeta,\\F\text{ indefinite}}} D^\mathrm{top}_{j,F}(\chi_F f|_{(U_F\setminus F)\cap \mathscr U_j^\zeta}).\label{eq:defDj}
\eq
The sum in \eqref{eq:defDj} is locally finite and $\chi_F f|_{(U_F\setminus F)\cap \mathscr U_j^\zeta}$ extends smoothly by zero to $\mathscr U_j^\zeta$, so that $D^\zeta_{j}$ is  well-defined. To see that 
\bq
\int_{\mathscr M^\zeta_\mathrm{top}}\langle D^\zeta_{0} a\rangle_T\d \mathscr M^\zeta_\mathrm{top}=2\pi\int_{\mathscr M^\zeta_\mathrm{top}}\langle a\rangle_T\d \mathscr M^\zeta_\mathrm{top}\qquad \forall\; a\in \CT(M),\label{eq:zeroorder}
\eq 
we observe that for each definite fixed point $F\in \F$ one has $U_F\cap \J^{-1}(\{\J(F)\})=F$ and consequently the intersection $\J^{-1}(\{\J(F)\})\cap M_{(\h_\mathrm{top})}$ is disjoint from $U_F$. This implies that for each $\zeta\in \t^\ast$ the family of cutoff functions $\{\chi_\mathrm{top},\chi_F\}_{F\in \F:\J(F)=\zeta, F\,\text{indefinite}}$ restricts to a partition of unity on the set $\J^{-1}(\{\zeta\})\cap M_{(\h_\mathrm{top})}$. Since 
$\mathscr D^\zeta_{0}=2\pi$ and $D^\mathrm{top}_{0,F}=2\pi$ for all indefinite $F\in \F$, we get \eqref{eq:zeroorder}.

Finally, to prove the naturality claim, let $(M',\omega',\J')$ be another Hamiltonian $T$-space and $\Phi:M\to M'$ an isomorphism of Hamiltonian $T$-spaces. Then $\F'=\{\Phi(F):F\in \F\}$ is the set of connected components of $M'^T$, and for each $F'=\Phi(F)\in \F'$ the map $\Phi_{F'}:=\Phi_F\circ\Phi^{-1}:U_F':=\Phi(U_F)\to P_F$ serves as the local normal form symplectomorphism in Proposition \ref{prop:localnormform}. Furthermore, the partition of unity $\{\chi_{\mathrm{top}}\circ \Phi^{-1},\chi_{\Phi^{-1}(F')}\circ \Phi^{-1}\}_{F'\in \F'}$ is subordinate to the cover $M'=M'_{(\h_{\mathrm{top})}}\cup \bigcup_{F'\in \F'} U_{F'}$, and one has $n^\pm_{F'}=n^\pm_{\Phi^{-1}(F')}$, $\lambda_j^{F'}=\lambda_j^{\Phi^{-1}(F')}$ for each $F'\in \F'$, $1\leq j \leq \mathrm{codim}\, F'/2=\mathrm{codim}\, \Phi^{-1}(F')/2$. The claim now follows from the construction of the operators in the proofs of Theorem \ref{thm:main1} and Corollary \ref{cor:geom1}, which is carried out locally either using Proposition \ref{prop:asymptreg}, for which the naturality property holds since the phase function on $M'$ is given by composition with $\Phi$ in the manifold variable, or by composing operators in a Euclidean space which are uniquely determined by the numbers $n^\pm_{F'}$ and $\lambda_j^{F'}$ with the local normal form symplectomorphism  $\Phi_{F'}$ and gluing them together using the partition of unity. This concludes the proof of Theorem \ref{thm:main1}. \\
\qed

The construction of $D_j^\zeta$ in \eqref{eq:defDj} raises the question whether the coefficients $(A^\zeta_j)_\mathrm{top}$ in \eqref{eq:distrib1} depend on the choice of partition of unity when $j>0$. That this is not the case is shown in the following 
\begin{lem}For each $j\in \N_0$, the distribution $\I_j^\zeta\in \D'(M)$ defined by $\I^\zeta_{j}(a):=\int_{\mathscr M^\zeta_\mathrm{top}}\langle D^\zeta_{j} a\rangle_T\d \mathscr M^\zeta_\mathrm{top}$, and consequently the coefficient $(A^\zeta_j)_\mathrm{top}$ in \eqref{eq:distrib1}, is independent of the choice of partition of unity.
\end{lem}
\begin{proof} Applying Theorem \ref{thm:main1} and Proposition \ref{prop:asymptreg} immediately yields that the restriction of $\I_j^\zeta$ to $M_{(\h_\mathrm{top})}$ is independent of the choice of partition of unity by the uniqueness of the coefficients in the respective asymptotic expansions. Now, given $j\in \N_0$, $\zeta\in \t^\ast$, and $a\in \CT(M)$, choose $\eps>0$ and a $T$-invariant cutoff function $\chi\in \CT(M)$  with $\mathrm{vol}_{\mathscr M^\zeta_\mathrm{top}}((\supp \chi)/T\cap \mathscr M^\zeta_\mathrm{top})<\eps/(\Vert D^\zeta_{j} a\Vert_\infty\vol T)$ and $\chi\equiv 1$ near $\supp a\cap M^T$. Then we have
\bqn
\I_j^\zeta(a)=\I_j^\zeta(\chi a)+ \I_j^\zeta((1-\chi)a)=\int_{\mathscr M^\zeta_\mathrm{top}}\langle\chi D^\zeta_{j} a\rangle_T\d \mathscr M^\zeta_\mathrm{top} + \I_j^\zeta(\chi^\zeta_j a) + \I_j^\zeta((1-\chi)a),
\eqn
where $\chi^\zeta_j\in \CT(M)$ is such that $D^\zeta_{j} (\chi a)=\chi D^\zeta_{j} a + \chi^\zeta_j a$. In particular, $\chi^\zeta_j$ vanishes near $\supp a\cap M^T$. Since $\chi^\zeta_j a$ and $(1-\chi)a$ are supported in $M_{(\h_\mathrm{top})}$, the terms $\I_j^\zeta(\chi^\zeta_j a)$ and $\I_j^\zeta((1-\chi)a)$ are independent of the choice of partition of unity. Finally, the integral of $\langle\chi D^\zeta_{j} a\rangle_T$ over $\mathscr M^\zeta_\mathrm{top}$ is bounded in absolute value by $\eps$. Since $\eps>0$ was arbitrary, $\I_j^\zeta(a)$ does not depend on the choice of partition of unity. 
\end{proof}

\subsection{Proof of Theorem \ref{thm:main2}}\label{sec:proof2} Fix a given $\zeta_0\in \R\cong \t^\ast$. To determine the limit behavior  of each individual term in \eqref{eq:distrib1} as $\zeta\to \zeta_0$ under the two conditions $\pm(\zeta-\zeta_0)>0$, we begin with the summands $A^\zeta_{j,F}(a\otimes \sigma)$, using the same notation as in the proof of Theorem \ref{thm:main1}. Thus, fix an $F\in \F$. 
As long as $\zeta \neq \J(F)$, we  have two options how to express the coefficient $A^\zeta_{j,F}(a\otimes \sigma)$: Either we can observe that $\zeta$ is a regular value of $\J|_{U_F}$,  and invoke the general regular stationary phase asymptotics from Proposition \ref{prop:asymptreg}, as we did in  (3) in the proof of Theorem \ref{thm:main1}. Or we can  apply Proposition \ref{prop:09.03.2020} (if $F$ is indefinite) or \ref{prop:4.3.def} (if $F$ is definite) and Corollary \ref{cor:geom1}, as we did in  (1) and (2) in the proof of Theorem \ref{thm:main1}. By uniqueness of the coefficients in asymptotic expansions, the two approaches do describe the same coefficients. 
However, only the latter approach is useful when $\zeta$ approaches the singular value $\J(F)$ of $\J|_{U_F}$ because the coefficients featured in Propositions \ref{prop:09.03.2020} and \ref{prop:4.3.def} do have a clearly visible limit behavior in this case, in contrast to the less explicit but simpler terms appearing in Proposition \ref{prop:asymptreg}, where no statement on limits towards singular values is made.  We thus distinguish three cases:
 \begin{enumerate}[leftmargin=*]
\item If $\J(F)=\zeta_0$ and $F$ is indefinite, we apply Proposition \ref{prop:09.03.2020} with $\eps$ replaced by $2\eps$, $\zeta$ by $2\zeta_F=2(\zeta-\J(F))$, $\mathscr{S}=\mathscr{S}_{\tilde a_F}$, and $L=\mathrm{codim}\,F/2-2$, together with Corollary \ref{cor:geom1}, yielding
\begin{align*}
\lim_{\substack{\zeta\to \J(F)\\
\pm(\zeta-\J(F))>0}} A^\zeta_{j,F}(a\otimes \sigma)
&=\begin{dcases}\sigma^{(j)}(0)\int_{\mathscr M^{\J(F)}_{\mathrm{top}}}\big<D^\mathrm{top}_{j,F}(\chi_F a)\big>_T\d \mathscr M^{\J(F)}_{\mathrm{top}},\qquad &j< j_F, \\
\sigma^{(j)}(0)\bigg(\int_{\mathscr M^{\J(F)}_{\mathrm{top}}}\big<D^\mathrm{top}_{j,F}(\chi_F a)\big>_T\d \mathscr M^{\J(F)}_{\mathrm{top}}+\int_{F}D^\mp_{j,F} a\d F\bigg), & j\geq j_F.\end{dcases}
\end{align*}
\item If $\J(F)=\zeta_0$ and $F$ is definite, we similarly get from Proposition \ref{prop:4.3.def} and Corollary \ref{cor:geom1}  the result
\begin{align*}
\lim_{\substack{\zeta\to \J(F)\\
\pm(\zeta-\J(F))>0}} A^\zeta_{j,F}(a\otimes \sigma)
&=\begin{dcases}0,\qquad &j< j_F\;\text{ or }\; s_F=\mp, \\
\sigma^{(j)}(0)\int_{F}D_{j,F} a, & j\geq j_F\;\text{ and }\; s_F=\pm.\end{dcases}
\end{align*}
\item If $\J(F)\neq\zeta_0$, then all $\zeta$ close to $\zeta_0$ are regular values of $\J|_{U_F}$, and by Proposition \ref{prop:asymptreg} one has
\bqn
\lim_{\substack{\zeta\to \zeta_0\\
\pm(\zeta-\zeta_0)>0}} A^\zeta_{j,F}(a\otimes \sigma)= A^{\zeta_0}_{j,F}(a\otimes \sigma)=\sigma^{(j)}(0)\int_{\mathscr M^{\zeta_0}_\mathrm{top}}\big<\mathscr D^{\zeta_0}_{j} (\chi_F a)\big>_T\d \mathscr M^{\zeta_0}_\mathrm{top}.
\eqn\end{enumerate}
Finally, it remains to consider the limits of the contributions $A^\zeta_{j,\mathrm{top}}(a\otimes \sigma)$ of the top stratum:
 \begin{enumerate}[leftmargin=*]
 \item[(4)] Since the function $\J|_{M_{\h_\mathrm{top}}}$ has only regular values, Proposition \ref{prop:asymptreg} gives 
\begin{align*}
\lim_{\substack{\zeta\to \zeta_0\\
\pm(\zeta-\zeta_0)>0}} A^\zeta_{j,\mathrm{top}}(a\otimes \sigma)&=A^{\zeta_0}_{j,\mathrm{top}}(a\otimes \sigma)=\sigma^{(j)}(0)\int_{\mathscr M^{\zeta_0}_\mathrm{top}}\big<\mathscr D^{\zeta_0}_{j} (\chi_\mathrm{top} a)\big>_T\d \mathscr M^{\zeta_0}_\mathrm{top}.
\end{align*}
\end{enumerate}
We are now ready to describe the limit behavior of
\bqn
A^\zeta_j(a\otimes\sigma)=A^\zeta_{j, \mathrm{top}}(a\otimes\sigma) + \sum_{F\in \F:F\cap \supp a\neq \emptyset}A^\zeta_{j,F}(a\otimes\sigma)
\eqn
 for each $\zeta\in \t^\ast$. Recalling the definition \eqref{eq:defDj} of the operators $D^\zeta_j$,  a combination of (1)--(4) yields  for each $j\in \N_0$   
\begin{align*}
&\lim_{\substack{\zeta\to \zeta_0\\
\pm(\zeta-\zeta_0)>0}} A^\zeta_j(a\otimes\sigma)=\lim_{\substack{\zeta\to \zeta_0\\
\pm(\zeta-\zeta_0)>0}}A^\zeta_{j,\mathrm{top}}(a\otimes\sigma) + \sum_{\substack{F\in \F: \J(F)\neq {\zeta_0},\\F\cap \supp a\neq \emptyset}}\lim_{\substack{\zeta\to \zeta_0\\
\pm(\zeta-\zeta_0)>0}}A^\zeta_{j,F}(a\otimes\sigma)\\
&\qquad\qquad\qquad\qquad\qquad\qquad\qquad\qquad\qquad\qquad + \sum_{\substack{F\in \F: \J(F)=\zeta_0,\\F\cap \supp a\neq \emptyset}}\lim_{\substack{\zeta\to \J(F)\\
\pm(\zeta-\J(F))>0}}A^\zeta_{j,F}(a\otimes\sigma)\\
&=\sigma^{(j)}(0)\bigg(\int_{\mathscr M^{\zeta_0}_{\mathrm{top}}}\big<D^{\zeta_0}_{j}a\big>_T\d \mathscr M^{\zeta_0}_{\mathrm{top}}+\sum_{\substack{F\in \F: \J(F)=\zeta_0,\\
\mathrm{codim}\,F/2-1\leq j,\\
F \text{ indefinite},\\
F\cap \supp a\neq \emptyset}} \int_{F}D^\mp_{j,F} a\d F+\sum_{\substack{F\in \F: \J(F)=\zeta_0,\\
\mathrm{codim}\,F/2-1\leq j,\\
F \text{ definite},\; s_F=\pm,\\
F\cap \supp a\neq \emptyset}} \int_{F}D_{j,F} a\d F\bigg).
\end{align*}
This concludes the proof of  Theorem \ref{thm:main2}. \\
\qed

\section*{Index of Notation}

In what follows we include a list with the main notation used in this paper, explaining its meaning and specifying the place where it is used first. 

\bigskip

\begin{longtable}{p{0.1\textwidth}p{0.7\textwidth}p{0.2\textwidth}}
 $M$ & A symplectic manifold with symplectic form $\omega$   & p.\ \pageref{Jintrod}\\
 $T$ & The circle group, acting on $M$ in a Hamiltonian fashion   & p.\ \pageref{Jintrod}\\
 $\t$ & The Lie algebra of $T$, identified with $\R$ by fixing a Lebesgue measure   & p.\ \pageref{Jintrod}\\
 $\J$ & The momentum map  $M\to \t^\ast$ & p.\ \pageref{Jintrod}\\
  $J(x)$ & The map $M \rightarrow \R$ given by $\J(p)(x)=J(x)(p)$ & Eq.\ \eqref{eq:duistermaatheckman}\\
$\S(V)$ & The space of Schwartz functions on $V$ & Eq.\ \eqref{eq:2}\\
$\D(M)$ & The space of test functions  $\CT(M)$ with the test function topology & Eq.\ \eqref{eq:Iasigmaeps}\\
$\D'(M)$ & The space of distributions on $M$, identified with the space of distribution densities on $M$ via the symplectic volume form $dM$ & Eq.\ \eqref{eq:exp000}\\
$\M^\zeta$ & The symplectic reduction $\M^\zeta=\J^{-1}(\{\zeta\})/T$ & Eq.\ \eqref{eq:mstrat}\\
$\mathscr M^\zeta_\mathrm{top}$& The top stratum of $\M^\zeta$ & Eq.\ \eqref{eq:mstrat}\\
$\mathscr M^\zeta_\mathrm{sing}$& The singular stratum of $\M^\zeta$ & Eq.\ \eqref{eq:mstrat}\\
$M_{(h_\aleph)}$ & The stratum of $M$ of infinitesimal orbit type $(\h_\aleph)$ &  Eq.\ \eqref{eq:mstrat} f.\\
$F$ & A connected component of the fixed point set $M^T$ & p.\ \pageref{Fintrod}\\
$\F$ & The set of all connected components $F$ of  $M^T$ & p.\ \pageref{Fintrod}\\
$I^\zeta(\varepsilon)$ & The generalized Witten integral & Eq. \eqref{eq:Iasigmaeps} \\
$a$ & A function in $\CT(M)$ & Eq. \eqref{eq:Iasigmaeps} \\
$\sigma$ & A function in $\S(\t)$ & Eq. \eqref{eq:Iasigmaeps} \\
$I_{a,\sigma}^\zeta(\eps)$ & The generalized Witten integral evaluated on $a \otimes \sigma$ & Eq. \eqref{eq:2} \\
$I_{\chi_\mathrm{top}}^\zeta (\varepsilon)$  & A component of the generalized Witten integral & Eq. \eqref{eq:sumdistr} \\ 
$I_{\chi_F}^\zeta (\varepsilon)$  & A component of the generalized Witten integral & Eq. \eqref{eq:sumdistr} \\ 
$A^\zeta_j$ & The $j$-th coefficient in the expansion of $I^\zeta(\varepsilon)$ & Thm. \ref{thm:main1} \\
$\Sigma^{\zetaF}$ & The local model  for the level set $\J^{-1}(\zeta+\J(F))$ & Def. \ref{def:quadrics} \\
$\Sigma^{\zetaF}_\bullet$ & A subset of $\Sigma^{\zetaF}$ & Def. \ref{def:quadrics} \\
$\Sigma^{\zetaF}_\times$ & A subset of $\Sigma^{\zetaF}_\bullet$ & Def. \ref{def:quadrics} \\
$A^\zeta_{j, \mathrm{top}}$ & A  contribution to the coefficient  $A^\zeta_j$ & Sec. \ref{sec:proof} \\
$A^\zeta_{j,F}$ & A contribution to the coefficient  $A^\zeta_j$ & Sec. \ref{sec:proof} \\
$a=\O(r)$ & There is a constant $C> 0$ such that $\lvert a\rvert\leq C r$ & Eq.\ \eqref{eq:bilineartrace} f.\\
$a=\O_x(r)$ & There is a constant $C> 0$, depending on $x$, such that $\lvert a\rvert\leq C r$ & Prop.\ \ref{prop:09.03.2020}
\end{longtable}

\bibliography{bibliography}
\bibliographystyle{amsplain}

\medskip

\end{document}